\documentclass{au}
 \presentedby{\dots}
 \received{\dots}{\dots}

\usepackage{amssymb, amsthm, amsmath}
\usepackage[mathscr]{eucal}
\usepackage{enumerate}
\usepackage{bbm} 

\usepackage{pgf}
\usepackage{tikz}
\usetikzlibrary{arrows,shapes,calc,backgrounds,fit}
\tikzset{%
 unshaded/.style={draw, shape=circle, fill=white, inner sep=1.5pt},
 shaded/.style={draw, shape=circle, fill=black, inner sep=1.5pt},
 invisible/.style={shape=circle, inner sep=1.5pt},
 label/.style={shape=rectangle, inner xsep=5pt, inner ysep=7pt},
 auto,
 curvy/.style={->, shorten >=3pt, shorten <=3pt, >=latex, looseness=1, bend angle=30},
 straight/.style={->, shorten >=3pt,shorten <=3pt, >=latex},
 loopy/.style={->, shorten >=3pt, shorten <=3pt, >=latex, min distance=15pt},
 order/.style={thin},
 bigloop/.style={-, min distance=65pt}}
 \pgfdeclarelayer{background}
 \pgfdeclarelayer{foreground}
 \pgfsetlayers{background,main,foreground}

\numberwithin{equation}{section}

\theoremstyle{plain}
\newtheorem{theorem}{Theorem}[section]
\newtheorem{lemma}[theorem]{Lemma}
\newtheorem{lemma:trans}[theorem]{Transfer Lemma}
\newtheorem{theorem:Ockham}[theorem]{Main Theorem}

\theoremstyle{definition}
\newtheorem{example}[theorem]{Example}
\newtheorem{note}[theorem]{Note}
\newtheorem{definition}[theorem]{Definition}

\newtheorem{remark}[theorem]{Remark}

\newcommand{\defn}[1]{\emph{#1}}

\newcommand{\Claim}[2]{\medskip\par\noindent\textit{#1:} {\emph{#2}}\par\smallskip\noindent}
\newcommand{\Case}[2]{\medskip\par\noindent\textit{#1:} \emph{#2}\par\smallskip\noindent}

\newcommand{\A}{{\mathbf A}}
\newcommand{\B}{{\mathbf B}}
\newcommand{\C}{{\mathbf C}}
\newcommand{\rb}{{\mathbf r}}

\newcommand{\KB}{{\mathbf K}}

\newcommand{\SB}{{\mathbf S}}
\newcommand{\TwB}{{\mathbf 2}}

\newcommand{\AT}{{\mathbb A}}
\newcommand{\BT}{{\mathbb B}}
\newcommand{\CT}{{\mathbb C}}
\newcommand{\DT}{{\mathbb D}}

\newcommand{\KT}{{\mathbb K}}
\newcommand{\ST}{{\mathbb S}}
\newcommand{\TwT}{{\mathbbm 2}}

\renewcommand{\P}{{\mathbb P}}

\newcommand{\X}{{\mathbb X}}
\newcommand{\Y}{{\mathbb Y}}
\newcommand{\Z}{{\mathbb Z}}

\newcommand{\T}{{\mathscr T}}

\newcommand{\cat}[1]{\boldsymbol{\mathscr{#1}}}
\newcommand{\CO}{{\cat O}}
\newcommand{\CY}{{\cat Y}}
\newcommand{\CD}{{\cat D}}
\newcommand{\CP}{{\cat P}}

\newcommand{\graph}[1]{\operatorname{graph}(#1)}
\newcommand{\id}[1]{\operatorname{id}_{#1}}
\newcommand{\Max}[1]{{#1_{\mathrm{max}}}}
\newcommand{\Syn}[1]{\operatorname{Syn}(#1)}
\newcommand{\End}[1]{\operatorname{End}(#1)}

\newcommand{\Var}[1]{\operatorname{Var}(#1)}
\newcommand{\HS}[1]{\mathsf{HS}(#1)}
\newcommand{\SH}[1]{\mathsf{SH}(#1)}
\newcommand{\IScP}[1]{\mathsf{ISP^{\scriptscriptstyle +}_f}(#1)}
\newcommand{\ISP}[1]{\mathsf{ISP}(#1)}
\newcommand{\IScPinf}[1]{\mathsf{IS_cP^{\scriptscriptstyle +}}(#1)}

\newcommand{\homsetrel}[2]{\mathrm{E}(#1)\rest{#2}}

\newcommand{\dotcup}{\mathbin{\dot{\cup}}}

\newcommand{\conj}{\And}
\newcommand{\bigconj}{\mathop{\text{\Large\&}}}

\newcommand{\abs}[1]{\lvert#1\rvert}
\newcommand{\comp}{{\setminus}}
\newcommand{\rest}[1]{{\upharpoonright}_{#1}}

\renewcommand{\le}{\leqslant}
\renewcommand{\ge}{\geqslant}
\newcommand{\nle}{\nleqslant}

\renewcommand{\emptyset}{\varnothing}
\renewcommand{\phi}{\varphi}
\renewcommand{\epsilon}{\varepsilon}

\newcommand{\pmods}[1]{\mkern8mu({\operatorname{mod}}\mkern 6mu#1)}

\begin{document}
\title[Counting relations on Ockham algebras]{Counting relations on Ockham algebras} 

\author[B. A. Davey]{Brian A. Davey} 
 \email{B.Davey@latrobe.edu.au}
 \address{Department of Mathematics and Statistics\\La Trobe University\\Victoria 3086\\
 Australia}

\author[L. T. Nguyen]{Long T. Nguyen} 
 \email{long\_nt@qtttc.edu.vn}
 \address{Department of Mathematics and Statistics\\La Trobe University\\Victoria 3086\\
 Australia}
 \curraddr{Quang Tri Teacher Training College\\Dong Ha City\\Quang Tri Province\\Vietnam}

\author[J. G. Pitkethly]{Jane G. Pitkethly} 
\email{J.Pitkethly@latrobe.edu.au}
 \address{Department of Mathematics and Statistics\\La Trobe University\\Victoria 3086\\
 Australia}



\dedicatory{Dedicated with best wishes from the second and third authors to\\ the first author, Brian Davey, on the occasion of his 65th birthday.}


\begin{abstract}
We find all finite Ockham algebras that admit only finitely many compatible relations (modulo a natural equivalence). Up to isomorphism and symmetry, these Ockham algebras form two countably infinite families: one family consists of the quasi-primal Ockham algebras, and the other family is a sequence of generalised Stone algebras.
\end{abstract}

\maketitle

\section{Introduction}

In this paper, we characterise the finite Ockham algebras that satisfy a very strong finiteness condition on their compatible relations. This condition is a natural strengthening of several well-known finiteness conditions.

An important example of a finiteness condition on a finite algebra $\A$ is that it is \defn{finitely related}:
\begin{itemize}
\item
there is a finite set $R$ of compatible relations on $\A$ such that each compatible relation on $\A$ can be defined from $R$ by a primitive-positive formula.
\end{itemize}
In particular, every finite algebra with a near-unanimity term is finitely related. This follows by Baker and Pixley's result~\cite{BP75} that a finite algebra $\A$ has a near-unanimity term if and only if the following condition holds:
\begin{itemize}
\item
there is a finite set $R$ of compatible relations on $\A$ such that each compatible relation on $\A$ can be defined from $R$ by a conjunction of atomic formulas.
\end{itemize}
The condition that we study in this paper is stronger again:
\begin{itemize}
\item
there is a finite set $R$ of compatible relations on $\A$ such that each compatible relation on $\A$ is interdefinable with a relation $r$ from $R$ via conjunctions of atomic formulas.
\end{itemize}
Such an algebra is said to \defn{admit only finitely many relations}. This condition was introduced by Davey and Pitkethly~\cite{DP10}, motivated by the study of alter egos in natural duality theory.

Since an algebra that admits only finitely many relations must have a near-unanimity term, it is natural to investigate the condition within varieties of lattice-based algebras. In this direction, the following results are known:
\begin{itemize}
\item
The finite Boolean algebras that admit only finitely many relations are those of size at most~$2$~(\cite{DP10}).
\item
The finite lattices that admit only finitely many relations are those of size at most~$2$~(\cite{DP10}).
\item
The finite Heyting algebras that admit only finitely many relations are the chains~(\cite{DP10,NP12}).
\end{itemize}
Our purpose in this paper is to add to this list by characterising the finite Ockham algebras that admit only finitely many relations.

Ockham algebras were introduced in 1977 by Berman~\cite{Ber77}. They have been studied by Urquhart~\cite{Urq79,Urq81}, Goldberg~\cite{Gol81,Gol83}, Blyth and Varlet~\cite{BV94}, and many others. An \defn{Ockham algebra} $\A=\langle A;\vee,\wedge,f,0,1 \rangle$ is bounded distributive lattice enriched with a unary operation $f$ that satisfies the equations $f(0) \approx 1$, $f(1) \approx 0$ and the familiar \defn{De Morgan Laws}:
\[
f(x \vee y) \approx f(x) \wedge f(y) \quad\text{and}\quad f(x \wedge y) \approx f(x) \vee f(y).
\]
The variety of Ockham algebras contains the varieties of Boolean algebras, Kleene algebras, De Morgan algebras, Stone algebras and MS-algebras.

Our characterisation is stated in terms of Ockham spaces. Priestley's duality for bounded distributive lattices~\cite{Pri70,Pri72} has a natural restriction to the variety of Ockham algebras~\cite{Urq79}: the dual space of a finite Ockham algebra $\A$ is a finite ordered set equipped with an order-reversing self-map~$g$. We will describe this duality in more detail in Section~\ref{sec:duality}.

For structures $\X$ and $\Y$ of the same type, we say that $\X$ is a \defn{divisor} of~$\Y$ if $\X \in \HS\Y$, that is, if $\X$ is a homomorphic image of a substructure of~$\Y$.

\begin{theorem:Ockham}\label{thm:main}
Let $\A$ be a non-trivial finite Ockham algebra. Then the following are equivalent:
\begin{enumerate}[ \normalfont(1)]
\item
$\A$ admits only finitely many relations;
\item
there is an odd number $m$ such that the dual space of~$\A$ is isomorphic to $\CT_m$, $\DT_m$ or $\DT_m^\partial$ from Figure~\ref{fig:finiteOckham};
\item
none of the eight Ockham spaces from Figure~\ref{fig:infiniteOckham} is a divisor of the dual space of~$\A$.
\end{enumerate}
\end{theorem:Ockham}

\begin{figure}[h!t]
\begin{tikzpicture}
   \begin{scope}
     \node[anchor=west] at (6.5,0) {$\CT_m$ ($m$ odd)};
     \node[unshaded] (0) at (0,0) {};
         \node[label, anchor=south] at (0) {$0$};
     \node[unshaded] (1) at (1,0) {};
         \node[label, anchor=south] at (1) {$1$};
     \node[unshaded] (2) at (2,0) {};
         \node[label, anchor=south] at (2) {$2$};
     \node[invisible] (3) at (3,0) {};
     \node at (3.5,0) {$\cdots$};
     \node[invisible] (4) at (4,0) {};
     \node[unshaded] (5) at (5,0) {};
         \node[label, anchor=south] at (5) {$m{-}1$};
     \draw[curvy] (0) to [bend left] (1);
     \draw[curvy] (1) to [bend left] (2);
     \draw[curvy] (2) to [bend left] (3);
     \draw[curvy] (4) to [bend left] (5);
     \draw[curvy, bend angle=20] (5) to [bend left] (0);
   \end{scope}
    \begin{scope}[yshift=-3.125cm]
     \node[anchor=west] at (6.5,0.75) {$\DT_m$ ($m$ odd)};
     \node[unshaded] (0) at (0,0) {};
         \node[label, anchor=east] at (0) {$0$};
     \node[unshaded] (m) at (0,1) {};
         \node[label, anchor=east] at (m) {$m$};
     \node[unshaded] (1) at (1,1) {};
         \node[label, anchor=north] at ($(1) + (0.1,0)$) {$1$};
     \node[unshaded] (2) at (2,1) {};
         \node[label, anchor=north] at (2) {$2$};
     \node[invisible] (3) at (3,1) {};
     \node at (3.5,1) {$\cdots$};
     \node[invisible] (4) at (4,1) {};
     \node[unshaded] (5) at (5,1) {};
         \node[label, anchor=north] at (5) {$m{-}1$};
     \draw[order] (0) to (m);
     \draw[curvy] (m) to [bend right] (1);
     \draw[curvy] (0) to [bend right] (1);
     \draw[curvy] (1) to [bend right] (2);
     \draw[curvy] (2) to [bend right] (3);
     \draw[curvy] (4) to [bend right] (5);
     \draw[curvy, bend angle=20] (5) to [bend right] (m);
     \end{scope}
   \begin{scope}[yshift=-5.5cm]
     \node[anchor=west] at (6.5,0.25) {$\DT_m^\partial$ ($m$ odd)};
     \node[unshaded] (0) at (0,1) {};
         \node[label, anchor=east] at (0) {$0$};
     \node[unshaded] (m) at (0,0) {};
         \node[label, anchor=east] at (m) {$m$};
     \node[unshaded] (1) at (1,0) {};
         \node[label, anchor=south] at ($(1) + (0.1,0)$) {$1$};
     \node[unshaded] (2) at (2,0) {};
         \node[label, anchor=south] at (2) {$2$};
     \node[invisible] (3) at (3,0) {};
     \node at (3.5,0) {$\cdots$};
     \node[invisible] (4) at (4,0) {};
     \node[unshaded] (5) at (5,0) {};
         \node[label, anchor=south] at (5) {$m{-}1$};
    \draw[order] (m) to (0);
    \draw[curvy] (m) to [bend left] (1);
    \draw[curvy] (0) to [bend left] (1);
    \draw[curvy] (1) to [bend left] (2);
    \draw[curvy] (2) to [bend left] (3);
    \draw[curvy] (4) to [bend left] (5);
    \draw[curvy, bend angle=20] (5) to [bend left] (m);
    \end{scope}
\end{tikzpicture}
\caption{The dual spaces of the non-trivial Ockham algebras with only finitely many relations}
\label{fig:finiteOckham}
\end{figure}
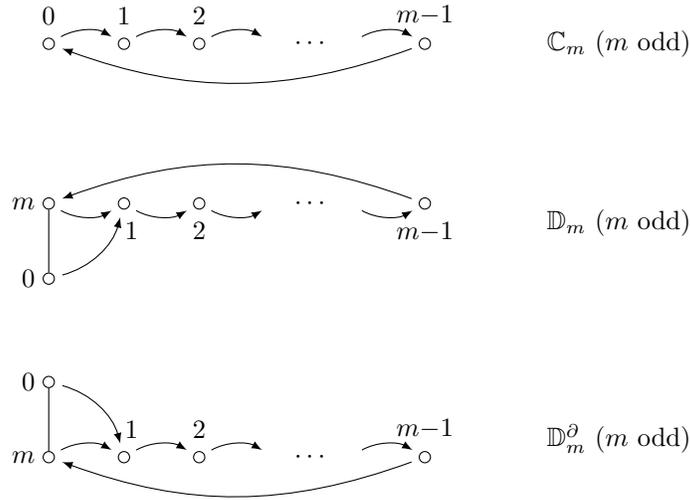

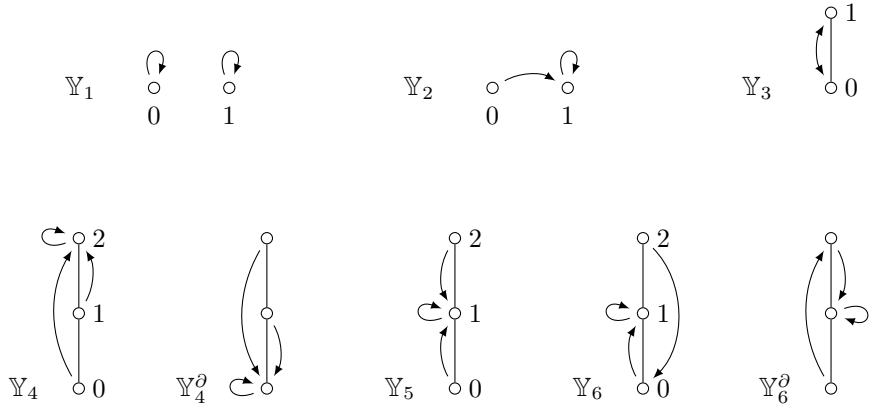
\begin{figure}[ht]
\begin{tikzpicture}
   \begin{scope}[xshift=0.5cm]
     \node[anchor=east] at (0.35,0) {$\Y_1$};
     \node[unshaded] (0) at (1,0) {};
       \node[label, anchor=north] at (0) {$0$};
     \node[unshaded] (1) at (2,0) {};
       \node[label, anchor=north] at (1) {$1$};
     \draw[loopy] (0) to [out=110,in=70] (0);
     \draw[loopy] (1) to [out=110,in=70] (1);
   \end{scope}
   \begin{scope}[xshift=5cm]
     \node[anchor=east] at (0.35,0) {$\Y_2$};
     \node[unshaded] (0) at (1,0) {};
       \node[label, anchor=north] at (0) {$0$};
     \node[unshaded] (1) at (2,0) {};
       \node[label, anchor=north] at (1) {$1$};
     \draw[loopy] (1) to [out=110,in=70] (1);
     \draw[curvy] (0) to [bend left](1);
     \end{scope}
   \begin{scope}[xshift=9.5cm]
     \node[anchor=east] at (0.35,0) {$\Y_3$};
     \node[unshaded] (0) at (1,0) {};
       \node[label, anchor=west] at (0) {$0$};
     \node[unshaded] (1) at (1,1) {};
       \node[label, anchor=west] at (1) {$1$};
     \draw[order] (0) to (1);
     \draw[curvy, <->] (0) to [bend left] (1);
     \end{scope}
   \begin{scope}[xshift=-0.5cm, yshift=-4cm]
     \node[anchor=east] at (0.6,0) {$\Y_4$};
     \node[unshaded] (0) at (1,0) {};
       \node[label, anchor=west] at (0) {$0$};
     \node[unshaded] (1) at (1,1) {};
       \node[label, anchor=west] at (1) {$1$};
     \node[unshaded] (2) at (1,2) {};
       \node[label, anchor=west] at (2) {$2$};
     \draw[order] (0) to (1);
     \draw[order] (1) to (2);
     \draw[loopy] (2) to [out=-160,in=-200] (2);
     \draw[curvy] (1) to [bend right] (2);
     \draw[curvy] (0) to [bend left] (2);
     \end{scope}
   \begin{scope}[xshift=2.00cm, yshift=-4cm]
     \node[anchor=east] at (0.35,0) {$\Y_4^\partial$};
     \node[unshaded] (0) at (1,0) {};
     \node[unshaded] (1) at (1,1) {};
     \node[unshaded] (2) at (1,2) {};
     \draw[order] (0) to (1);
     \draw[order] (1) to (2);
     \draw[loopy] (0) to [out=-160,in=-200] (0);
     \draw[curvy] (1) to [bend left] (0);
     \draw[curvy] (2) to [bend right] (0);
     \end{scope}
   \begin{scope}[xshift=4.5cm, yshift=-4cm]
     \node[anchor=east] at (0.6,0) {$\Y_5$};
     \node[unshaded] (0) at (1,0) {};
       \node[label, anchor=west] at (0) {$0$};
     \node[unshaded] (1) at (1,1) {};
       \node[label, anchor=west] at (1) {$1$};
     \node[unshaded] (2) at (1,2) {};
       \node[label, anchor=west] at (2) {$2$};
     \draw[order] (0) to (1);
     \draw[order] (1) to (2);
     \draw[loopy] (1) to [out=-160,in=-200] (1);
     \draw[curvy] (0) to [bend left] (1);
     \draw[curvy] (2) to [bend right] (1);
     \end{scope}
   \begin{scope}[xshift=7.00cm, yshift=-4cm]
     \node[anchor=east] at (0.6,0) {$\Y_6$};
     \node[unshaded] (0) at (1,0) {};
       \node[label, anchor=west] at (0) {$0$};
     \node[unshaded] (1) at (1,1) {};
       \node[label, anchor=west] at (1) {$1$};
     \node[unshaded] (2) at (1,2) {};
       \node[label, anchor=west] at (2) {$2$};
     \draw[order] (0) to (1);
     \draw[order] (1) to (2);
     \draw[loopy] (1) to [out=-160,in=-200] (1);
     \draw[curvy] (0) to [bend left] (1);
     \draw[curvy, bend angle=45] (2) to [bend left] (0);
     \end{scope}
   \begin{scope}[xshift=9.5cm, yshift=-4cm]
     \node[anchor=east] at (0.6,0) {$\Y_6^\partial$};
     \node[unshaded] (0) at (1,0) {};
     \node[unshaded] (1) at (1,1) {};
     \node[unshaded] (2) at (1,2) {};
     \draw[order] (0) to (1);
     \draw[order] (1) to (2);
     \draw[loopy] (1) to [out=20,in=-20] (1);
     \draw[curvy] (0) to [bend left] (2);
     \draw[curvy] (2) to [bend left] (1);
     \end{scope}
\end{tikzpicture}
\caption{The eight dual-space obstacles}\label{fig:infiniteOckham}
\end{figure}

Up to symmetry, the non-trivial finite Ockham algebras that admit only finitely many relations can be grouped into the following two infinite families.
\begin{itemize}
\item
The Ockham algebras with dual spaces $\CT_1$, $\CT_3$, $\CT_5$, $\dots$: The first member of this family is the $2$-element Boolean algebra, which has dual space~$\CT_1$. In Section~\ref{sec:quasi}, we will show that this family consists precisely of the quasi-primal Ockham algebras. It is known that every quasi-primal algebra admits only finitely many relations~\cite[2.10]{DPW13}.
\item
The Ockham algebras with dual spaces $\DT_1$, $\DT_3$, $\DT_5$, $\dots$: The first member of this family is the $3$-element Stone algebra, which has dual space~$\DT_1$. For each odd number~$m$, let $\SB_m$ denote the Ockham algebra with dual space~$\DT_m$. In Section~\ref{sec:fin}, we will give a natural duality for the variety generated by~$\SB_m$, and see that it mimics very closely the well-known natural duality for Stone algebras~\cite{Dav78,Dav82}. We use this duality to represent the compatible relations on $\SB_m$ and thereby show that $\SB_m$ admits only finitely many relations.
\end{itemize}

Our characterisation for Ockham algebras in general can easily be restricted to yield characterisations within familiar subvarieties. For example, the variety of MS-algebras~\cite{BV83} (which includes both De Morgan algebras and Stone algebras) consists of all Ockham algebras with dual spaces satisfying $x \le g^2(x)$. The only Ockham spaces in Figure~\ref{fig:finiteOckham} that satisfy this condition are $\CT_1$ and~$\DT_1$, and so the only non-trivial finite MS-algebras that admit only finitely many relations are the $2$-element Boolean algebra and the $3$-element Stone algebra.

\section{Background: compatible relations}\label{sec:comp}

This section introduces some basic definitions and results concerning the equivalence of compatible relations. By way of example, we 
first consider two compatible relations on the  $2$-element bounded lattice $\TwB = \langle \{0, 1\}; \vee, \wedge, 0, 1 \rangle$. Define
\[
{\le} := \{00, 01, 11\} \subseteq \{0, 1\}^2
 \ \text{ and }\ 
\rho := \{0000, 0100, 0011, 0111, 1111\} \subseteq \{0, 1\}^4,
\]
as in Figure~\ref{fig:equivalent}.

\begin{figure}[ht]
\begin{tikzpicture}
\begin{scope}[xshift=0.5cm] 
  \node[anchor=east] at (-0.5,0) {${\le}$};
  \node[unshaded] (0) at (0,0) {};
    \node[label, anchor=west] at (0) {$00$};
  \node[unshaded] (01) at (0,1) {};
    \node[label, anchor=west] at (01) {$01$};
  \node[unshaded] (11) at (0,2) {};
    \node[label, anchor=west] at (11) {$11$};
  \draw[order] (0) to (01);
  \draw[order] (01) to (11);
\end{scope}
\begin{scope}[xshift=6cm]
  \node[anchor=east] at (-2,0) {$\rho$};
  \node[unshaded] (0) at (0,0) {};
    \node[label, anchor=west] at (0)  {$0000$};
  \node[unshaded] (a) at ($(0) + (135:1)$) {};
    \node[label, anchor=east] at (a)  {$0100$};
  \node[unshaded] (b) at ($(0) + (45:1)$) {};
    \node[label, anchor=west] at (b)  {$0011$};
  \node[unshaded] (c) at ($(a) + (45:1)$) {};
    \node[label, anchor=west] at (c)  {$0111$};
  \node[unshaded] (1) at ($(c) + (0,1)$) {};
    \node[label, anchor=west] at (1)  {$1111$};
  \draw[order] (0) to (a);
  \draw[order] (0) to (b);
  \draw[order] (a) to (c);
  \draw[order] (b) to (c);
  \draw[order] (c) to (1);
\end{scope}
\end{tikzpicture}
\caption{Two equivalent compatible relations on $\TwB$}
\label{fig:equivalent}
\end{figure}
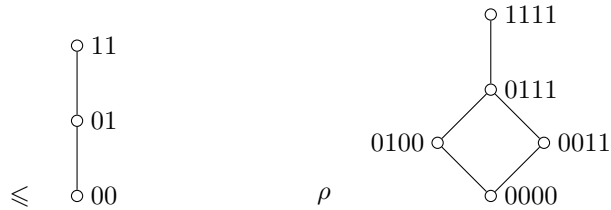

We regard the relations $\le$ and $\rho$ as equivalent, since they are interdefinable as follows:
\begin{align*}
{\le} &= \{\,(a, b) \in \{0, 1\}^2 \mid (a, a, b, b) \in \rho \,\},\\
\rho &= \{\,(a, b, c, d) \in \{0, 1\}^4 \mid a \le b \conj a \le c \conj c = d \,\}.
\end{align*}
In fact, every compatible relation on $\TwB$ is interdefinable in this way with either $\le$ or the unary relation $\{0,1\}$. Hence there is a natural sense in which the $2$-element bounded lattice $\TwB$ has only two compatible relations.

\begin{definition}
Let $A$ be a non-empty finite set, and consider relations $r \subseteq A^k$ and $s \subseteq A^\ell$, for some $k, \ell \ge 1$. Then we say that $r$ is \defn{conjunct-atomic definable} from~$s$ if we can write
\[
r = \bigl\{\, (a_1,\dotsc,a_k) \in A^k \bigm| \textstyle\bigconj_{i=1}^n \Phi_i(a_1,\dotsc,a_k) \,\bigr\},
\]
for some $n \ge 0$, where each $\Phi_i(x_1,\dotsc,x_k)$ is an atomic formula in~$s$. We say that the two relations $r$ and $s$ are \defn{equivalent} if each is conjunct-atomic definable from the other.
\end{definition}

\begin{definition}
Now let $\A$ be a finite algebra. For each $k \ge 1$, a relation $r \subseteq A^k$ is \defn{compatible} with~$\A$ if it is a non-empty subuniverse of~$\A^k$. 
We say that $\A$~\defn{admits only finitely many relations} if the set of compatible relations on~$\A$ has a finite number of equivalence classes (modulo conjunct-atomic interdefinability); otherwise, we say that $\A$~\defn{admits infinitely many relations}.

\end{definition}

The following lemma will help us to find Ockham algebras that admit infinitely many relations, by giving a sense in which this property is `contagious'.

\begin{lemma:trans}[{\cite[3.3]{DP10}}]\label{lem:transfer}
Let\/ $\A$ and $\B$ be finite algebras such that\/ $\A$ is a divisor of\/~$\B$. If\/ $\A$ admits infinitely many relations, then so does~$\B$.
\end{lemma:trans}

The next two lemmas will help with finding Ockham algebras that admit only finitely many relations. First, we define a relation $r$ on $A$ to be \defn{directly decomposable} if, up to permutation of coordinates, it can be written as $p \times q$, for some non-trivial relations $p$ and $q$ on~$A$. Otherwise, the relation $r$ is \defn{directly indecomposable}. The following lemma is implicit in the proof of~\cite[2.10]{DPW13}.

\begin{lemma}\label{lem:indec}
Let\/ $\A$ be a finite algebra. Then $\A$ admits only finitely many relations if and only if the set of all directly indecomposable compatible relations on\/~$\A$ has a finite number of equivalence classes \textup(modulo conjunct-atomic interdefinability\textup).
\end{lemma}

\begin{definition}\label{def:alterego}
Let $\A$ be a finite algebra and let $\AT = \langle A; G, H, R \rangle$ be a
structure on the same underlying set, where
\begin{itemize}
\item $G$ is a set of finitary operations on $A$,
\item $H$ is a set of finitary partial operations on $A$, and
\item $R$ is a set of finitary relations on $A$.
\end{itemize}
Then $\AT$ is an \defn{alter ego} of $\A$ if each relation in 
$R \cup \{\, \graph f \mid f \in G \cup H \,\}$ is compatible with~$\A$. We use $\IScP{\AT}$ to denote the class of all isomorphic copies of non-empty substructures of non-zero finite powers of~$\AT$.
\end{definition}

If $\AT$ is an alter ego of~$\A$, then for any structure $\X$ in $\IScP{\AT}$ and any non-empty subset $S$ of~$X$, we can define an $S$-ary compatible relation on $\A$ by
\[
\homsetrel \X {S} := \bigl\{\, \alpha \rest {S} \bigm| \alpha \colon \X \to \AT \text{ is a morphism} \,\bigr\} \subseteq A^S.
\]
A basic result of clone theory states that, if the alter ego $\AT$ determines the clone of~$\A$, then every compatible relation on $\A$ is equivalent to one of the form $\homsetrel {\AT^n} {S}$, where $S$ is a non-empty subset of~$A^n$.

We can obtain a tighter description of the compatible relations on $\A$ by assuming that the alter ego $\AT$ satisfies the following interpolation condition:
\begin{enumerate}[({I}C)]
\item
for all $n \ge 1$ and all $\X \le \AT^n$, every morphism $\alpha \colon \X \to \AT$ extends to an $n$-ary term function of the algebra~$\A$.
\end{enumerate}

\begin{lemma}[{\cite[2.3]{DP10}}]\label{lem:compat}
Let\/ $\A$ be a finite algebra and let\/ $\AT$ be an alter ego of\/ $\A$ such that\/ \textup{(IC)} holds. Then every compatible relation on\/ $\A$ is equivalent to one of the form\/ $\homsetrel \X {S}$, where $\X \in \IScP\AT$ and\/ $S$ is a non-empty generating set for\/~$\X$.
\end{lemma}

The interdefinability of relations $\homsetrel \X {S}$ and $\homsetrel \Y {T}$ can be interpreted as a condition on maps between the structures $\X$, $\Y$ and~$\AT$. This leads to general techniques for showing that an algebra $\A$ admits only finitely many relations (see Lemmas~\ref{lem:con1} and~\ref{lem:con2}) or infinitely many relations (see Lemma~\ref{lem:inf}).

\begin{note}
In the theory of natural dualities, an alter ego $\AT$ of a finite algebra $\A$ is equipped with the discrete topology. In this paper, we mostly work within the class $\IScP{\AT}$, where topology plays no role: each structure in this class is finite and so the inherited topology is discrete. We will include the topology only in Section~\ref{sec:fin}, where we consider natural dualities.
\end{note}

\section{Background: Ockham algebras}\label{sec:duality}

This section gives a brief introduction to the restricted Priestley duality for Ockham algebras. Recall that $\A = \langle A; \vee,\wedge,f,0,1\rangle$ is an \defn{Ockham algebra} if
\begin{itemize}
\item
$\A^\flat = \langle A; \vee,\wedge,0,1\rangle$ is a bounded distributive lattice, and
\item
$f$ is a dual endomorphism of~$\A^\flat$.
\end{itemize}
Figure~\ref{fig:familiar} gives several important examples of Ockham algebras: the subdirectly irreducible generators of the subvarieties of Boolean algebras, Kleene algebras, De Morgan algebras, Stone algebras and MS-algebras.

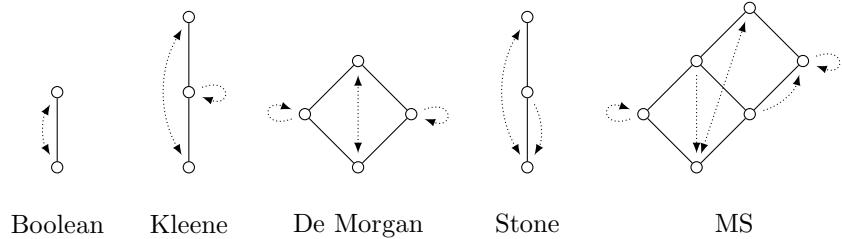
\begin{figure}[ht]
\begin{center}
\begin{tikzpicture}
\begin{scope}
  \node[anchor=north] at (0,-0.5) {Boolean};
  \node[unshaded] (0) at (0,0) {};
  \node[unshaded] (1) at (0,1) {};
  \draw[order] (0) to (1);
  \draw[curvy, <->, densely dotted] (0) to [bend left] (1);
\end{scope}
\begin{scope}[xshift=1.75cm]
  \node[anchor=north] at (0,-0.5) {Kleene};
  \node[unshaded] (0) at (0,0) {};
  \node[unshaded] (a) at (0,1) {};
  \node[unshaded] (1) at (0,2) {};
  \draw[order] (0) to (a);
  \draw[order] (a) to (1);
  \draw[curvy, <->, densely dotted] (0) to [bend left] (1);
  \draw[loopy, densely dotted] (a) to [out=20,in=-20] (a);
\end{scope}
\begin{scope}[xshift=4.0cm]
  \node[anchor=north] at (0,-0.5) {De Morgan};
  \node[unshaded] (0) at (0,0) {};
  \node[unshaded] (a) at ($(0) + (135:1)$) {};
  \node[unshaded] (b) at ($(0) + (45:1)$) {};
  \node[unshaded] (1) at ($(a) + (45:1)$) {};
  \draw[order] (0) to (a);
  \draw[order] (0) to (b);
  \draw[order] (a) to (1);
  \draw[order] (b) to (1);
  \draw[straight, <->, densely dotted] (0) to (1);
  \draw[loopy, densely dotted] (a) to [out=-160,in=-200] (a);
  \draw[loopy, densely dotted] (b) to [out=20,in=-20] (b);
\end{scope}
\begin{scope}[xshift=6.25cm]
  \node[anchor=north] at (0,-0.5) {Stone};
  \node[unshaded] (0) at (0,0) {};
  \node[unshaded] (a) at (0,1) {};
  \node[unshaded] (1) at (0,2) {};
  \draw[order] (0) to (a);
  \draw[order] (a) to (1);
  \draw[curvy, <->, densely dotted] (0) to [bend left] (1);
  \draw[curvy, densely dotted] (a) to [bend left] (0);
\end{scope}
\begin{scope}[xshift=8.5cm]
  \node[anchor=north] at (0.5,-0.5) {MS};
  \node[unshaded] (0) at (0,0) {};
  \node[unshaded] (a) at ($(0) + (135:1)$) {};
  \node[unshaded] (b) at ($(0) + (45:1)$) {};
  \node[unshaded] (c) at ($(a) + (45:1)$) {};
  \node[unshaded] (d) at ($(b) + (45:1)$) {};
  \node[unshaded] (1) at ($(c) + (45:1)$) {};
  \draw[order] (0) to (a);
  \draw[order] (0) to (b);
  \draw[order] (a) to (c);
  \draw[order] (b) to (c);
  \draw[order] (b) to (d);
  \draw[order] (c) to (1);
  \draw[order] (d) to (1);
  \draw[loopy, densely dotted] (a) to [out=-160,in=-200] (a);
  \draw[straight, densely dotted] (c) to (0);
  \draw[curvy, densely dotted] (b) to [bend right] (d);
  \draw[loopy, densely dotted] (d) to [out=20,in=-20] (d);
  \draw[straight, <->, densely dotted] (0) to (1);
\end{scope}
\end{tikzpicture}
\caption{Some subdirectly irreducible Ockham algebras}
\label{fig:familiar}
\end{center}
\end{figure}

An \defn{Ockham space} is a topological structure $\X = \langle X; g, \le, \T \rangle$ such that
\begin{itemize}
\item
$\X^\flat = \langle X; \le, \T \rangle$ is a \defn{Priestley space} (that is, an ordered compact topological space such that, for all $x, y \in X$ with $x \nle y$, there is a clopen down-set $V$ with $x \notin V$ and $y \in V$), and
\item
$g$ is a dual endomorphism of $\X^\flat$ (that is, a continuous order-reversing self-map on~$X$).
\end{itemize}
We shall use $\CO$ and $\CY$ to denote the categories of Ockham algebras and Ockham spaces, respectively. The morphisms of $\CO$ are the Ockham-algebra homomorphisms, and the morphisms of $\CY$ are the continuous order-preserving maps that also preserve the unary operation~$g$.
These two categories are dually equivalent (Urquhart~\cite{Urq79}), with the associated contravariant functors $H \colon \CO \to \CY$ and $K \colon \CY \to \CO$ given on objects as follows.

\begin{definition}
Let $\TwB = \langle \{0,1\}; \vee, \wedge, 0, 1\rangle$ be the $2$-element bounded lattice, and let $\CD$ be the category of bounded distributive lattices. For each Ockham algebra~$\A$, define the Ockham space
\[
H(\A) = \langle \CD(\A^\flat, \TwB); g, \le, \T \rangle,
\]
where $\langle \CD(\A^\flat, \TwB); \le, \T \rangle$ is the Priestley space dual to the bounded distributive lattice~$\A^\flat$ and the unary operation $g$ is given by $g(x) = (x \circ f)'$, for all $x \colon \A^\flat \to \TwB$. Here $'$ denotes the usual Boolean complement on $\{0, 1\}$.

Now let $\TwT = \langle \{0,1\}; \le, \T\rangle$ be the $2$-element chain equipped with the discrete topology, and let $\CP$ be the category of Priestley spaces. For each Ockham space~$\X$, define the Ockham algebra
\[
K(\X) = \langle \CP(\X^\flat, \TwT); \vee, \wedge, f, 0, 1  \rangle,
\]
where $\langle \CP(\X^\flat, \TwT); \vee, \wedge, 0, 1 \rangle$ is the bounded distributive lattice dual to the Priestley space~$\X^\flat$ and the unary operation $f$ is given by $f(\alpha) = (\alpha \circ g)'$, for all $\alpha \colon \X^\flat \to \TwT$.
\end{definition}

We can now finish setting up the duality for Ockham algebras in the natural way. In particular, we have the following definitions.

\begin{definition}\
\begin{itemize}
\item
For each homomorphism $\varphi \colon \A \to \B$ in $\CO$, define $H(\varphi) \colon H(\B) \to H(\A)$ by $H(\varphi)(x) := x \circ \varphi$, for all $x \colon \B^\flat \to \TwB$.
\item
For each morphism $\psi \colon \X \to \Y$ in $\CY$, define $K(\psi) \colon K(\Y) \to K(\X)$ by $K(\psi)(\alpha) := \alpha \circ \psi$, for all $\alpha \colon \Y^\flat \to \TwT$.
\item
For each Ockham algebra~$\A$, the isomorphism $e_\A \colon \A \to KH(\A)$ is given by $e_\A(a)(x) = x(a)$, for all $a \in A$ and $x \colon \A^\flat \to \TwB$.
\item
For each Ockham space~$\X$, the isomorphism $\varepsilon_\X \colon \X \to HK(\X)$ is given by $\varepsilon_\X(x)(\alpha) = \alpha(x)$, for all $x \in X$ and $\alpha \colon \X^\flat \to \TwT$.
\end{itemize}
\end{definition}

This duality for Ockham algebras restricts naturally to the five subvarieties from Figure~\ref{fig:familiar}: the descriptions of the dual spaces are summarised in Table~\ref{tab:familiar} (see~\cite{CF77,Dav82,BV90}).

\begin{table}[ht]
\caption{Dual spaces for familiar subvarieties of Ockham algebras}
\label{tab:familiar}
\renewcommand{\arraystretch}{1.2}
\setlength{\tabcolsep}{1em}
\begin{tabular}{ll}
  \hline
  Subvariety & Dual spaces \\\hline
  Boolean  & $g(x) \approx x$ \\
  Kleene  & $x$ and $g(x)$ are comparable, and $g^2(x) \approx x$  \\
  De Morgan  & $g^2(x) \approx x$ \\
  Stone  & $g(x)$ is the unique maximal above $x$ \\
  MS  & $x \le g^2(x)$ \\
  \hline
\end{tabular}
\end{table}

We finish this section by proving two basic facts about the duality for Ockham algebras that will be needed in later sections. The following is the natural restriction of the corresponding result from Priestley duality (see~\cite[7.4.1]{CD98}).

\begin{lemma}\label{lem:strong}
A homomorphism $\varphi \colon \A \to \B$ in $\CO$ is an embedding \textup(respectively, a surjection\textup) if and only if its dual morphism $H(\varphi) \colon H(\B) \to H(\A)$ in $\CY$ is a surjection \textup(respectively, an embedding\textup).
\end{lemma}

This lemma leads easily to the following result, which will allow us to apply the Transfer Lemma~\ref{lem:transfer} to Ockham algebras from within the dual class $\CY$ of Ockham spaces.

\begin{lemma}\label{lem:divisor}
Let\/ $\A$ and\/ $\B$ be Ockham algebras. Then $\A$ is a divisor of\/ $\B$ if and only if\/ $H(\A)$ is a divisor of\/ $H(\B)$.
\end{lemma}
\begin{proof}
Assume that $\A \in \HS \B$. Then there exists $\C \in \CO$ with an embedding $\varphi \colon \C \hookrightarrow \B$
and a surjection $\psi \colon \C \twoheadrightarrow \A$. By Lemma~\ref{lem:strong}, there is a surjection $H(\varphi) \colon H(\B) \twoheadrightarrow H(\C)$ and an embedding $H(\psi) \colon H(\A) \hookrightarrow H(\C)$. 
Thus $H(\A) \in \SH {H(\B)} \subseteq \HS {H(\B)}$.

Now assume that $H(\A) \in \HS {H(\B)}$. As the categories $\CO$ and $\CY$ are dually equivalent, there exists $\C \in \CO$ with an embedding $H(\varphi) \colon H(\C) \hookrightarrow H(\B)$ and a surjection $H(\psi) \colon H(\C) \twoheadrightarrow H(\A)$. Using Lemma~\ref{lem:strong} again, we have $\varphi \colon \B \twoheadrightarrow \C$ and $\psi \colon \A \hookrightarrow \C$. So $\A \in \SH {\B} \subseteq \HS {\B}$.
\end{proof}

The following lemma, which is used in the next section, generalises part of Lemma~\ref{lem:strong} (cf.~\cite[7.4.1]{CD98}).

\begin{lemma}\label{lem:joint}
A homomorphism $\varphi \colon \A \to \B_1 \times \B_2$ in $\CO$ is an embedding if and only if the two morphisms $H(\pi_i \circ \varphi) \colon H(\B_i) \to H(\A)$ in $\CY$, for $i \in \{1, 2\}$, are jointly surjective.
\end{lemma}
\begin{proof}
For $i \in \{1,2\}$, let $\pi_i \colon \B_1 \times \B_2 \to \B_i$ be the $i$th projection. Then the map $H(\pi_1) \dotcup H(\pi_2) \colon H(\B_1) \dotcup H(\B_2) \to H(\B_1 \times \B_2)$ is an isomorphism, as pairwise coproducts in $\CY$ are given by disjoint union. Let $\varphi_i := \pi_i \circ \varphi \colon \A \to \B_i$. Then $H(\varphi_i) = H(\pi_i \circ \varphi) = H(\varphi) \circ H(\pi_i)$. 
So the following diagram commutes.
\begin{center}
\begin{tikzpicture}
  \node[rectangle] (0) at (0,0) {$H(\B_1) \dotcup H(\B_2)$};
  \node[rectangle] (1) at (0,1.75) {$H(\B_1 \times \B_2)$};
  \node[rectangle] (2) at (3,1.75) {$H(\A)$};
  \path (0) edge [>->>] node {\small $H(\pi_1) \dotcup H(\pi_2)$} (1);
  \path (0) edge [->,anchor=north west] node[yshift=0.1cm] {\small $H(\varphi_1) \dotcup H(\varphi_2)$} (2);
  \path (1) edge [->] node {\small $H(\varphi)$} (2);
\end{tikzpicture}
\end{center}
%
Thus $H(\varphi)$ is surjective if and only if $H(\varphi_1)$ and $H(\varphi_2)$ are jointly surjective. The claim follows because $H(\varphi)$ is surjective if and only if $\varphi$ is an embedding, by Lemma~\ref{lem:strong}.
\end{proof}

\section{Quasi-primal Ockham algebras}\label{sec:quasi}

In this section, we show that an Ockham algebra is quasi-primal if and only if its dual space is isomorphic to~$\CT_m$ from Figure~\ref{fig:finiteOckham}, for some odd~$m$.

We will use the following description of the binary compatible relations on an Ockham algebra. Similar results have been used many times in the literature; see, for example, \cite[3.3]{DP93}, \cite[3.5]{DP96}, \cite[p.\ 218]{CD98} and 
\cite[pp.\ 222--223]{DH04}. We include a proof for completeness.

\begin{lemma}\label{lem:binary}
Let\/ $\X$ be an Ockham space and let\/ $\rb \le K(\X)^2$. Then there exist jointly surjective morphisms $\varphi_1, \varphi_2 \colon \X \to H(\rb)$ such that
\[
r = \{\, (\alpha \circ \varphi_1, \alpha \circ \varphi_2) \mid \alpha \in KH(\rb) \,\}.
\]
\end{lemma}
\begin{proof}
Define $\A := K(\X)$. Then $\rb \le \A^2$. Let $\rho_1,\rho_2 \colon \rb \to \A$ denote the two projections. The inclusion $\rho_1 \sqcap \rho_2 \colon \rb \to \A^2$ is an embedding, and therefore the morphisms $H(\rho_1), H(\rho_2) \colon H(\A) \to H(\rb)$ are jointly surjective by Lemma~\ref{lem:joint}.

As $\A = K(\X)$, for each $i \in \{1,2\}$ we can define $\varphi_i \colon \X \to H(\rb)$ by ${\varphi_i := H(\rho_i) \circ \varepsilon_{\X}}$. 
Since $\varepsilon_\X \colon \X \to H(\A)$ is an isomorphism, the morphisms $\varphi_1, \varphi_2 \colon \X \to H(\rb)$ are jointly surjective.

Since $e_\rb \colon \rb \to KH(\rb)$ is an isomorphism, we have
\begin{align*}
 r &= \bigl\{\, \bigl(\rho_1(a), \rho_2(a)\bigr) \bigm| a \in r \,\bigr\}\\
   &= \bigl\{\, \bigl( \rho_1 \circ e^{-1}_\rb (\alpha),\,  \rho_2 \circ e^{-1}_\rb (\alpha) \bigr) \bigm| \alpha \in KH(\rb) \,\bigr\}.
\end{align*}
So it remains to check that $\rho_i \circ e^{-1}_\rb (\alpha) = \alpha\circ \varphi_i $, for each $i \in \{1,2\}$ and $\alpha \in KH(\rb)$. Since $\langle H, K, e, \varepsilon \rangle$ is a dual adjunction between the categories $\CO$ and~$\CY$, we have $\rho_i = K(H(\rho_i) \circ \varepsilon_{\X}) \circ e_\rb$; see~\cite[Figure 1.2]{CD98}. Thus
\[
\rho_i \circ e^{-1}_\rb (\alpha)
 = K(H(\rho_i) \circ \varepsilon_{\X})(\alpha)
 = \alpha \circ H(\rho_i) \circ \varepsilon_{\X}
 = \alpha \circ \varphi_i,
\]
as required.
\end{proof}

We next prove some basic facts about cycles in Ockham spaces that will be used in this section and in the final section.

\begin{definition}
Let $\X = \langle X; g, \le, \T\rangle$ be an Ockham space and let $C \subseteq X$. For $m \ge 1$, we will say that $C$ is an \defn{$m$-cycle} of $\X$ if we can enumerate $C$ as $c_0,\dots,c_{m-1}$ such that $g(c_i) = c_{i+1 \pmod m}$. In this case, we say that $C$ is an \defn{odd cycle} if $m$ is odd, and an \defn{even cycle} otherwise.  Note that a $1$-cycle of $\X$ is just a fixpoint of~$g$.
\end{definition}

\begin{lemma}\label{lem:oddcycle}
Let\/ $\X$ be an Ockham space such that every element belongs to an odd cycle. Then $\X$ is an antichain.
\end{lemma}
\begin{proof}
Let $c,d \in X$ with $c$ in an $m$-cycle and $d$ in an $n$-cycle, for some odd $m$ and~$n$. Assume that $c \le d$ in~$\X$. As $m$ and $n$ are odd and $g$ is order-reversing, we have
\[
g^m(d) \le g^m(c) = c \quad\text{and}\quad d = g^n(d) \le g^n(c).
\]
As $m + n$ is even, it now follows that
\[
d \le g^n(c) = g^{m+n}(c) \le g^{m + n}(d) = g^m(d) \le c.
\]
Thus $c = d$. Hence $\X$ is an antichain.
\end{proof}

\begin{lemma}\label{lem:evencycle}
Let\/ $\X$ be an Ockham space that contains an even cycle. Then the Ockham space $\Y_3$ from Figure~\ref{fig:infiniteOckham} is a divisor of\/~$\X$.
\end{lemma}
\begin{proof}
Assume $\CT \le \X$ such that $C$ is an $m$-cycle, for some even~$m$. Since $C$ is finite and $g\rest C \colon C \to C$ is an order-reversing bijection, it follows that $g\rest C$ is a dual order-automorphism of~$\CT$. So $g$ sends maximal elements of $C$ to minimal elements of~$C$, and vice versa. Let $x$ be a maximal element of~$C$. Then we have $C = \{\, g^k(x) \mid k \in \{0,1,\dotsc,m-1\} \,\}$. As $m$ is even, we can define $\varphi \colon \CT \twoheadrightarrow \Y_3$ by
\[
\varphi(g^k(x)) =
\begin{cases}
1, &\text{if $k$ is even,}\\
0, &\text{if $k$ is odd.}
\end{cases}
\]
So $\Y_3 \in \HS\X$, as required.
\end{proof}

\begin{theorem}\label{lem:quasiOckham}
An Ockham algebra is quasi-primal if and only if its dual space is isomorphic to~$\CT_m$ from Figure~\ref{fig:finiteOckham}, for some odd~$m$.
\end{theorem}
\begin{proof}
Let $m$ be odd. We first show that the Ockham algebra $\A := K(\CT_m)$ is quasi-primal. As $\A$ is lattice-based and therefore has a 
ternary near-unanimity term, it suffices to show that every subalgebra of $\A^2$ is either the product of two subalgebras of $\A$ or the graph of a partial automorphism of $\A$ (by~\cite[3.3.12]{CD98}).

Let $\rb \le \A^2$. Then, using Lemma~\ref{lem:binary}, there are jointly surjective morphisms $\varphi_1, \varphi_2 \colon \CT_m \to H(\rb)$ such that
\[
r =  \{\, ( \alpha \circ \varphi_1, \alpha \circ \varphi_2 ) \mid \alpha \in KH(\rb) \,\}.
\]
Since $C_m$ is an odd cycle and $\varphi_1, \varphi_2 \colon \CT_m \to H(\rb)$ are jointly surjective, it follows that the Ockham space $H(\rb)$ is either an odd cycle or the union of two different odd cycles. We consider these two cases separately.

\Case{Case 1} {$H(\rb)$ is an odd cycle.}%
In this case, the morphism $\varphi_i \colon \CT_m \to H(\rb)$ is surjective, for each $i \in \{1,2\}$. We will show that $r$ is the graph of a partial automorphism of~$\A$. Let $a,b,c \in A$ with $(a,b), (a,c) \in r$. Then there exist $\beta, \gamma \in KH(\rb)$ such that $a = \beta \circ \varphi_1$, $b = \beta \circ \varphi_2$ and $a = \gamma \circ \varphi_1$, $c = \gamma \circ \varphi_2$. So $\beta \circ \varphi_1 = a = \gamma \circ \varphi_1$. Since $\varphi_1$ is surjective, we must have $\beta = \gamma$ and hence $b = c$. By symmetry, if $(a,b), (c,b) \in r$, then $a = c$.

\Case{Case 2} {$H(\rb)$ is the union of two different odd cycles.}%
By Lemma~\ref{lem:oddcycle}, the Ockham space $H(\rb)$ is an antichain. So we can write $H(\rb) = \X_1 \dotcup \X_2$, where $\X_i$ is an odd cycle with $\varphi_i \colon \CT_m \to \X_i$. It follows that
\begin{align*}
 r &= \{\, ( \alpha \circ \varphi_1, \alpha \circ \varphi_2 ) \mid \alpha \in KH(\rb) \,\}\\
   &=\{\, \alpha_1 \circ \varphi_1 \mid \alpha_1 \in K(\X_1) \,\} \times \{\, \alpha_2 \circ \varphi_2 \mid \alpha_2 \in K(\X_2) \,\}.
\end{align*}
So $r$ is the product of two subalgebras of~$\A$.

Now assume that $\A$ is a quasi-primal Ockham algebra. Then $\A$ is simple. By Lemma~\ref{lem:strong}, this implies that the dual space $H(\A)$ has no non-empty proper substructures. So $H(\A)$ must be an $m$-cycle, for some $m \ge 1$. Suppose that $m$ is even. Then $\Y_3$ is a divisor of~$H(\A)$, by Lemma~\ref{lem:evencycle}. Thus the $3$-element Kleene algebra $\KB = K(\Y_3)$ belongs to the variety generated by~$\A$, by Lemma~\ref{lem:divisor}. But $\KB$ generates the variety of all Kleene algebras, which is not congruence permutable. This contradicts our assumption that $\A$ is quasi-primal. Therefore $m$ is odd. So $H(\A)$ is isomorphic to~$\CT_m$, by Lemma~\ref{lem:oddcycle}.
\end{proof}

Every quasi-primal algebra admits only finitely many relations~\cite[2.10]{DPW13}. So the Ockham algebra with dual space $\CT_m$ admits only finitely many relations, for each odd~$m$. We obtain an alternative proof of this in the next section.

\section{Ockham algebras with finitely many relations}\label{sec:fin}

In this section, we prove the implication $(2) \Rightarrow (1)$ in our main theorem. Note that $\CT_m$ is a divisor of~$\DT_m$, for each odd~$m$. Using symmetry and Lemmas~\ref{lem:transfer} and~\ref{lem:divisor}, the implication $(2) \Rightarrow (1)$ will follow directly once we show that the Ockham algebra $\SB_m := K(\DT_m)$ admits only finitely many relations, for each odd~$m$.

\begin{example}
The Ockham algebra $\SB_1 = K(\DT_1)$ is the $3$-element Stone algebra.
Figure~\ref{fig:stone} shows the Ockham space~$\DT_1$ and the corresponding Ockham algebra~$\SB_1$, where each element $\alpha$ of $S_1 = \CP(\DT_1^\flat,\TwT)$ is written as the string $\alpha(0)\alpha(1)$.

The natural duality for the variety of Stone algebras~\cite{Dav78,Dav82} is based on $\SB_1$ and its alter ego $\ST_1 := \langle S_1; u, \preccurlyeq, \T \rangle$, where the unary operation~$u$ and the order relation~$\preccurlyeq$ are given in Figure~\ref{fig:stone}; see~\cite[4.3.6]{CD98}. The dual class $\IScPinf {\ST_1}$ consists of all Priestley spaces $\X = \langle X; u, \preccurlyeq, \T \rangle$ with a continuous self-map $u$ that sends each element up to the unique maximal above it.
\end{example}

\begin{figure}[t]
\begin{tikzpicture}
    \begin{scope}
     \node[anchor=north] at (0,-1) {$\DT_1 = \langle \{0,1\}; g, \le, \T \rangle$};
     \node[unshaded] (0) at (0,1) {};
         \node[label, anchor=west] at (0) {$1$};
     \node[unshaded] (1) at (0,0) {};
         \node[label, anchor=west] at (1) {$0$};
     \draw[order] (0) to (1);
     \draw[loopy] (0) to [out=110,in=70] (0);
     \draw[curvy] (1) to [bend left] (0);
     \end{scope}
    \begin{scope}[xshift=4.0cm]
     \node[anchor=north] at (0,-1) {$\SB_1 = \langle S_1; \vee, \wedge, f, 0, 1 \rangle$};
     \node[unshaded] (0) at (0,0) {};
         \node[label, anchor=west] at (0) {$00$};
     \node[unshaded] (a) at (0,1) {};
         \node[label, anchor=west] at (a) {$01$};
     \node[unshaded] (1) at (0,2) {};
         \node[label, anchor=west] at (1) {$11$};
     \draw[order] (0) to (a);
     \draw[order] (a) to (1);
     \draw[curvy, densely dotted, <->] (0) to [bend left] (1);
     \draw[curvy, densely dotted] (a) to [bend left] (0);
     \end{scope}
    \begin{scope}[xshift=7.5cm]
     \node[anchor=north] at (0.5,-1) {$\ST_1 = \langle S_1; u, \preccurlyeq, \T \rangle$};
     \node[unshaded] (0) at (0,1) {};
         \node[label, anchor=east] at (0) {$00$};
     \node[unshaded] (a) at (1,0) {};
         \node[label, anchor=west] at (a) {$01$};
     \node[unshaded] (1) at (1,1) {};
         \node[label, anchor=west] at (1) {$11$};
     \draw[order] (a) to (1);
     \draw[loopy, densely dashed] (0) to [out=110,in=70] (0);
     \draw[loopy, densely dashed] (1) to [out=110,in=70] (1);
     \draw[curvy, densely dashed] (a) to [bend left] (1);
     \end{scope}
\end{tikzpicture}
\caption{The $3$-element Stone algebra $\SB_1$ with its dual space $\DT_1$ and alter ego $\ST_1$}
\label{fig:stone}
\end{figure}

We will generalise this duality by using Davey and Werner's piggyback technique~\cite{DW85,DW86} to give a natural duality for the variety generated by~$\SB_m$, for each odd~$m$. We want to find an alter ego $\ST_m$ that strongly dualises~$\SB_m$, as it will follow automatically that (IC) holds, and so we can then use Lemma~\ref{lem:compat} to describe the compatible relations on~$\SB_m$.

Davey and Priestley~\cite{DP87} generalised the basic piggyback technique to obtain a two-sorted natural duality for the variety generated by any finite subdirectly irreducible Ockham algebra; see~\cite[7.5.5]{CD98}. We use the following simple version of the piggyback technique to show how this duality simplifies in a special case.

\begin{theorem}[Piggyback Duality Theorem \cite{DW85,DW86}]\label{thm:pb}
Let $\A$ be a finite algebra that has a bounded distributive lattice~$\A^\flat$ as a reduct. Assume there is a homomorphism $\omega \colon \A^\flat \to \TwB$ and a set\/ $G \subseteq \End\A$ such that
\begin{enumerate}[\ \normalfont(S)]
\item 
for all distinct $a, b \in A$, there exists $e \in G$ with $\omega(e(a)) \ne \omega(e(b))$.
\end{enumerate}
Let $R$ be the set of all compatible binary relations on $\A$ that are maximal in
\[
\omega^{-1}(\le) := \{\, (a,b) \in A^2 \mid \omega(a) \le \omega(b) \,\}.
\]
Then the alter ego $\AT = \langle A; G, R, \T \rangle$ of $\A$ yields a duality on $\ISP\A$.
\end{theorem}

\begin{note}
A finite Ockham algebra $\A$ is subdirectly irreducible if and only if its dual space~$\X$ is one-generated (Urquhart~\cite{Urq79}). In this case, we have $\Var\A = \ISP\A$ if and only if $g$ is order-preserving or $g^2(X) = g(X)$, and then $\A$ is injective in $\Var\A$; see~\cite[3.10]{DP87} and~\cite[4.17]{Gol81}.
\end{note}

Since the Ockham space $\DT_m$ from Figure~\ref{fig:finiteOckham} is one-generated with order-preserving~$g$, the following theorem applies.

\begin{theorem}\label{thm:duality}
Let\/ $\X$ be a finite one-generated Ockham space, with generator~$0$, and assume that\/ $g$ is order-preserving. Define the Ockham algebra $\A = K(\X)$ and the alter ego $\AT = \langle \CP(\X^\flat,\TwT); u, \preccurlyeq, \T \rangle$, where
\begin{itemize}
\item
$u$ is the endomorphism $K(g)$ of~$\A$, and
\item
$\preccurlyeq$ is the \emph{alternating order}
on $A$ given by $\alpha \preccurlyeq \beta$ if and only if
\begin{align*}
\alpha(g^k(0)) \le \beta(g^k(0)), &\text{ for all even\/ }k \ge 0,\\
&\text{and\/ } \alpha(g^k(0)) \ge \beta(g^k(0)), \text{for all odd\/ } k \ge 1.
\end{align*}
\end{itemize}
Then $\AT$ yields a strong duality on $\Var\A$, and so \textup{(IC)} holds.
\end{theorem}
\begin{proof}
We want to apply Theorem~\ref{thm:pb}. Define $\omega \colon \A^\flat \to \TwB$ by $\omega(\alpha) := \alpha(0)$, for all $\alpha \in A = \CP(\X^\flat,\TwT)$. Consider $\alpha \ne \beta$ in~$A$. Since $\X$ is generated by~$0$, there is some $k \ge 0$ with $\alpha(g^k(0)) \ne \beta(g^k(0))$. Since $g \in \End\X$, we have $u := K(g) \in \End\A$ with $\omega(u^k(\alpha)) = \alpha\circ g^k(0) \ne \beta\circ g^k(0) = \omega(u^k(\beta))$. Thus condition~(S) holds.

As $\A$ is an Ockham algebra, there is a unique compatible binary relation $r$ on~$\A$ that is maximal in $\omega^{-1}(\le)$, given by
\[
(\alpha, \beta) \in r \ \iff\ (f^k(\alpha), f^k(\beta)) \in \omega^{-1}(\le), \text{ for all } k \ge 0;
\]
see Davey and Priestley~\cite[3.5]{DP87}. For $k$ even, we have $f^k(\alpha) = \alpha \circ g^k$, and for $k$~odd, we have $f^k(\alpha) = (\alpha \circ g^k)'$. So it follows that $r = {\preccurlyeq}$.

By Theorem~\ref{thm:pb}, the alter ego $\AT' := \langle \CP(\X^\flat,\TwT); \End\A, \preccurlyeq, \T \rangle$ yields a duality on $\ISP\A = \Var\A$. Since $\A$ is injective in $\Var\A$, each partial endomorphism of $\A$ extends to an endomorphism of~$\A$. Every non-trivial subalgebra of $\A$ is subdirectly irreducible, so $\A$ has irreducibility index~$1$ and it follows by general results~\cite[3.3.7, 3.2.3(iii)]{CD98} that $\AT'$ yields a strong duality on $\Var\A$.

To complete the proof, it remains to check that $\End\A$ is generated by~$u$. Each endomorphism of $\A$ is of the form $K(e)$, for some $e \in \End\X$. Since $e(0) = g^k(0)$, for some $k \ge 0$, it follows that $e = g^k$ and so $K(e) = u^k$.
\end{proof}

For the Ockham algebra $\SB_m$ with dual space~$\DT_m$, the definition of the alter ego given in the previous theorem simplifies as follows.

\begin{definition}\label{defn:STm}
For $m$ odd, define the \defn{alternating alter ego} of the Ockham algebra $\SB_m$ with dual space $\DT_m$ to be the structure
\[
\ST_m := \langle \CP(\DT_m^\flat,\TwT); u, \preccurlyeq, \T \rangle
\]
from Theorem~\ref{thm:duality}, so that $u \in \End\A$ is given by $u(\alpha) := \alpha \circ g$, and the order $\preccurlyeq$ is given by
\[
\alpha \preccurlyeq \beta \ \iff\ \alpha(0) \le \beta(0) \text{ and } \alpha \rest {D_m \comp \{0\}} = \beta \rest  {D_m \comp \{0\}}.
\]
For example, the alternating alter ego $\ST_1$ of $\SB_1$ agrees with the familiar alter ego from Figure~\ref{fig:stone}, and the alternating alter ego $\ST_3$ of $\SB_3$ is shown in Figure~\ref{fig:ST3}, where each element $\alpha$ of $S_3 = \CP(\DT_3^\flat,\TwT)$ is written as the string $\alpha(0)\alpha(1)\alpha(2)\alpha(3)$.
\end{definition}

\begin{figure}[t]
\begin{tikzpicture}
   \begin{scope}
      \node[unshaded] (0) at (0,1.2) {};
         \node[label, anchor=south] at (0) {$\scriptstyle 0000$};
      \node[unshaded] (1) at (1.5,1.2) {};
         \node[label, anchor=south] at (1) {$\scriptstyle 1111$};
      \node[unshaded] (c) at (1.5,0) {};
         \node[label, anchor=north] at (c) {$\scriptstyle 0111$};
      \node[unshaded] (k) at (3,1.2) {};
         \node[label, anchor=south] at (k) {$\scriptstyle 1001$};
      \node[unshaded] (i) at (3,0) {};
         \node[label, anchor=north] at (i) {$\scriptstyle 0001$};
      \node[unshaded] (j) at (4.5,1.2) {};
         \node[label, anchor=south] at (j) {$\scriptstyle 0010$};
      \node[unshaded] (d) at (6,1.2) {};
         \node[label, anchor=south] at (d) {$\scriptstyle 0100$};
      \node[unshaded] (h) at (7.5,1.2) {};
         \node[label, anchor=south] at (h) {$\scriptstyle 1101$};
      \node[unshaded] (e) at (7.5,0) {};
         \node[label, anchor=north] at (e) {$\scriptstyle 0101$};
      \node[unshaded] (a) at (9,1.2) {};
         \node[label, anchor=south] at (a) {$\scriptstyle 1011$};
      \node[unshaded] (g) at (9,0) {};
         \node[label, anchor=north] at (g) {$\scriptstyle 0011$};
      \node[unshaded] (b) at (10.5,1.2) {};
         \node[label, anchor=south] at (b) {$\scriptstyle 0110$};
      \draw[order] (c) to (1);
      \draw[order] (i) to (k);
      \draw[order] (e) to (h);
      \draw[order] (g) to (a);
      \draw[loopy, densely dashed] (0) to [out=-160,in=-200] (0);
      \draw[curvy, densely dashed] (c) to [bend left] (1);
      \draw[loopy, densely dashed] (1) to [out=-160,in=-200] (1);
      \draw[straight, densely dashed] (k) to [bend right] (j);
      \draw[straight, densely dashed] (j) to [bend right] (d);
      \draw[straight, densely dashed] (i) to [bend right] (j);
      \draw[curvy, densely dashed] (d) to [bend right] (k);
      \draw[straight, densely dashed] (h) to [bend right] (a);
      \draw[straight, densely dashed] (e) to [bend right] (a);
      \draw[straight, densely dashed] (a) to [bend right] (b);
      \draw[straight, densely dashed] (g) to [bend right] (b);
      \draw[curvy, densely dashed] (b) to [bend right] (h);
      \end{scope}
\end{tikzpicture}
\caption{The alternating alter ego $\ST_3 = \langle \CP(\DT_3^\flat,\TwT); u, \preccurlyeq, \T \rangle$ of~$\SB_3$}
\label{fig:ST3}
\end{figure}

Since $\ST_m$ satisfies the interpolation condition (IC) with respect to $\SB_m$, it follows from Lemma~\ref{lem:compat} that every compatible relation on $\SB_m$ is equivalent to one of the form $\homsetrel \X S$, for some $\X \in \IScP{\ST_m}$ and generating set $S$ for~$\X$. To be able to make use of this description of the compatible relations on~$\SB_m$, we next develop an intrinsic description of the topological structures in the dual class~$\IScPinf{\ST_m}$. We shall use the following lemma.

\begin{lemma}\label{lem:sepstructs}
Let\/ $m$ be odd, and let\/ $\ST_m = \langle S_m; u, \preccurlyeq, \T\rangle$ be the alternating alter ego of the Ockham algebra~$\SB_m$.
\begin{enumerate}[ \normalfont(1)] 
\item
The structure $\Z_0 = \langle \{a, 1\}; u, \preccurlyeq, \T\rangle$ 
shown below embeds into~$\ST_m$.
\item
For each divisor $k$ of\/~$m$, the structure\/ $\Z_k = \langle \{0, a_0, a_1, \dots, a_{k-1}\}; u, \preccurlyeq, \T\rangle$ 
shown below embeds into~$\ST_m$.
\end{enumerate}
\end{lemma}

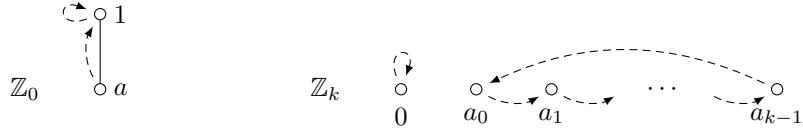
\begin{center}
\begin{tikzpicture}
   \begin{scope}
     \node at (-1,0) {$\Z_0$};
     \node[unshaded] (1) at (0,1) {};
         \node[label, anchor=west] at (1) {$1$};
     \node[unshaded] (a) at (0,0) {};
         \node[label, anchor=west] at (a) {$a$};
     \draw[order] (a) to (1);
     \draw[curvy, densely dashed] (a) to [bend left] (1);
     \draw[loopy, densely dashed] (1) to [out=-160,in=-200] (1);
     \end{scope}
     \begin{scope}[xshift=4cm]
     \node at (-1,0) {$\Z_k$};
     \node[unshaded] (0) at (0,0) {};
         \node[label, anchor=north] at (0) {$0$};
     \node[unshaded] (a0) at (1,0) {};
         \node[label, anchor=north] at (a0) {$a_0$};
     \node[unshaded] (a1) at (2,0) {};
         \node[label, anchor=north] at (a1) {$a_1$};
     \node[invisible] (a2) at (3,0) {};
     \node at (3.5,0) {$\cdots$};
     \node[invisible] (a3) at (4,0) {};
     \node[unshaded] (ak-1) at (5,0) {};
         \node[label, anchor=north] at (ak-1) {$a_{k-1}$};
      \draw[curvy, densely dashed] (a0) to [bend right] (a1);
      \draw[curvy, densely dashed] (a1) to [bend right] (a2);
      \draw[curvy, densely dashed] (a3) to [bend right] (ak-1);
      \draw[curvy, bend angle=25, densely dashed] (ak-1) to [bend right] (a0);
      \draw[loopy, densely dashed] (0) to [out=110,in=70] (0);
      \end{scope}
\end{tikzpicture}
\end{center}

\begin{proof}
(1) Recall that $S_m$ consists of all order-preserving maps from $\DT_m$ to $\TwT$. Let $\underline 1 \colon \DT_m^\flat \to \TwT$ be the constant map onto~$1$, and define $\alpha \colon \DT_m^\flat \to \TwT$ by
\begin{equation*}
 \alpha(i) =
 \begin{cases}
 0, &\text{if $i = 0$,}\\
 1, &\text{otherwise.}
 \end{cases}
\end{equation*}
Using Definition~\ref{defn:STm}, we see that $\alpha \preccurlyeq \underline 1$ and $u(\alpha) = \underline 1 = u(\underline 1)$. So $\{ \alpha, \underline 1 \}$ forms a substructure of $\ST_m$ isomorphic to~$\Z_0$.

(2) Let $k$ be a divisor of $m$. We will prove that $\Z_k$ embeds into $\ST_m$. For each $j \in \{0, 1, \dots, k-1\}$, define $\alpha_j \colon D_m \to \{0,1\}$ by
\begin{equation*}
 \alpha_j(i) =
 \begin{cases}
 1, &\text{if $i \equiv j \pmods k$,}\\
 0, &\text{otherwise.}
 \end{cases}
\end{equation*}
Then $\alpha_j(m) = \alpha_j(0)$, as $m \equiv 0 \pmod k$, and it follows that $\alpha_j$ is order-preserving. Now let $\underline 0 \colon D_m \to \{0,1\}$ be the constant map onto~$0$. We want to prove that the set $X := \{\underline 0, \alpha_0, \alpha_1, \dots, \alpha_{k-1}\}$ forms the following substructure of~$\ST_m$.

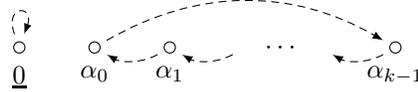
\begin{center}
\begin{tikzpicture}
    \begin{scope}
     \node[unshaded] (0) at (0,0) {};
         \node[label, anchor=north] at (0) {$\underline 0$};
     \node[unshaded] (a0) at (1,0) {};
         \node[label, anchor=north] at (a0) {$\alpha_0$};
     \node[unshaded] (a1) at (2,0) {};
         \node[label, anchor=north] at (a1) {$\alpha_1$};
     \node[invisible] (a2) at (3,0) {};
     \node at (3.5,0) {$\cdots$};
     \node[invisible] (ak-2) at (4,0) {};
     \node[unshaded] (ak-1) at (5,0) {};
         \node[label, anchor=north] at (ak-1) {$\alpha_{k-1}$};
     \draw[loopy, densely dashed] (0) to [out=110,in=70] (0);
     \draw[curvy, densely dashed] (a1) to [bend left] (a0);
     \draw[curvy, densely dashed] (a2) to [bend left] (a1);
     \draw[curvy, densely dashed] (ak-1) to [bend left] (ak-2);
     \draw[curvy, densely dashed] (a0) to [bend left] (ak-1);
     \end{scope}
\end{tikzpicture}
\end{center}

We first check the order relation $\preccurlyeq$ on~$X$. The maps $\alpha_0$ and $\underline 0$ are incomparable, since $\alpha_0(m) = 1$. For $j \in \{1, \dots, k-1\}$, the maps $\alpha_j$ and $\underline 0$ are incomparable, since $\alpha_j(j) = 1$. Finally, for distinct $j, \ell \in \{0, 1, \dots, k-1\}$, we have $\alpha_j(j) = 1 \nle 0 = \alpha_\ell(j)$ and so $\alpha_j \not\preccurlyeq \alpha_\ell$. Thus $X$ is an antichain.

We now check the action of $u$ on $X$. Clearly, we have $u(\underline 0) = \underline 0$. Now let $j \in \{0, 1, \dots, k-1\}$. We want to show that $u(\alpha_j) = \alpha_\ell$, where $\ell \equiv j - 1 \pmod k$. For each $i \in D_m$, we have
\begin{align*}
 u(\alpha_j)(i) = \alpha_j(g(i)) &=
 \begin{cases}
 1, &\text{if $g(i) \equiv j \pmods k$,}\\ 
 0, &\text{otherwise,}
 \end{cases}\\
 &=
 \begin{cases}
 1, &\text{if $i + 1 \equiv j \pmods k$,}\\ 
 0, &\text{otherwise,}
 \end{cases}\\
 & = \alpha_\ell(i),
\end{align*}
as required. It follows that $X = \{\underline 0, \alpha_0, \alpha_1, \dots, \alpha_{k-1}\}$ forms a substructure of $\ST_m$ isomorphic to~$\Z_k$.
\end{proof}

We can now give an intrinsic description of the topological structures in the dual class $\IScPinf{\ST_m}$.

\begin{theorem}\label{thm:topological}
Let\/ $m$ be odd, and let\/ $\X = \langle X; u, \preccurlyeq, \T \rangle$ be a topological structure of the same type as~$\ST_m$. Then $\X \in \IScPinf{\ST_m}$ if and only if
\begin{enumerate}[ \normalfont(1)] 
\item
$\langle X; \preccurlyeq, \T \rangle$ is a Priestley space,
\item
$\X$ satisfies $x \preccurlyeq y \implies u(x) \approx u(y)$, and
\item
$\X$ satisfies $x \preccurlyeq u^m(x)$.
\end{enumerate}
Moreover, it follows from conditions \textup{(1)--(3)} that
\begin{enumerate}[ \normalfont(1)]\setcounter{enumi}{3}
\item 
each $\preccurlyeq$-connected component of\/ $\X$ has a greatest element,
\item 
$u$ sends each element of\/ $X$ to a maximal element, and
\item 
an element $x$ of\/ $X$ is maximal if and only if\/ $u^m(x) = x$.
\end{enumerate}
\end{theorem}

\begin{proof}
First assume that $\X \in \IScPinf{\ST_m}$. As $\ST_m = \langle S_m; u, \preccurlyeq, \T \rangle$ has an underlying Priestley space, it follows that $\X$ does too. Since conditions~(2) and~(3) are quasi-atomic formulas, we can show they hold in~$\X$ by showing they hold in~$\ST_m$.

Let $\alpha, \beta \in S_m = \CP(\DT_m^\flat, \TwT)$ with $\alpha \preccurlyeq \beta$. Then $\alpha \rest {D_m \comp \{0\}} = \beta \rest {D_m \comp \{0\}}$. Since $g(D_m) \subseteq D_m \comp \{0\}$, we get $u(\alpha) = \alpha \circ g = \beta \circ g = u(\beta)$. Thus $\ST_m$ satisfies~(2).

Let $\alpha \in S_m$. Note that $u^m(\alpha) = \alpha \circ g^m$. Since $\alpha$ is order-preserving, we have $\alpha(0) \le \alpha(m) = \alpha(g^m(0))$. Since $g^m$ fixes each element in $D_m \comp \{0\}$, it follows that $\alpha \preccurlyeq \alpha \circ g^m = u^m(\alpha)$. Thus $\ST_m$ satisfies~(3), and so $\X$ satisfies (1)--(3).

Now assume that $\X$ satisfies \textup{(1)--(3)}. We first show that $\X$ also satisfies (4)--(6). Note that, as $\X$ is a Priestley space, every element of $X$ is less than or equal to a maximal element. Assume that $x$ and $y$ are
maximal elements in the same $\preccurlyeq$-connected component of~$\X$. Then $u(x) = u(y)$, by~(2), and so $u^m(x) = u^m(y)$. As $x$ and $y$ are maximal, 
(3) gives $x = u^m(x) = u^m(y) = y$. This establishes~(4).

Now let $x, y \in X$ with $u(x) \preccurlyeq y$. Then $u^{m+1}(x) = u^m(y)$, by~(2). Using~(3) and~(2) together gives $u(x) = u^{m+1}(x)$. By~(3), it follows that $y \preccurlyeq u^m(y) = u^{m+1}(x) = u(x)$. So $y = u(x)$, and therefore $u(x)$ is a maximal. Thus condition~(5) holds. Condition~(6) follows easily from (3) and~(5).

Let $\Max X$ be the set of maximal elements of~$\X$. Then $\Max X = u(X)$, by~(5) and~(6). Since $\X$ is compact Hausdorff and $u \colon X \to X$ is continuous, the map $u$ is closed (see~\cite[B.1]{CD98}). Hence $\Max X = u(X)$ is a closed subset of~$X$.

We now prove that $\X \in \IScPinf{\ST_m}$. It suffices to find enough morphisms from $\X$ to $\ST_m$ to `separate' the order relation~$\preccurlyeq$. Let $x,y \in X$ with $x \not\preccurlyeq y$. We want to find a morphism $\phi \colon \X \to \ST_m$ such that $\phi(x) \not\preccurlyeq \phi(y)$. We shall consider two cases.

\Case{Case 1} {$y \notin \Max X$.}%
We will use the substructure $\Z_0$ of $\ST_m$ from Lemma~\ref{lem:sepstructs}. Since $\X$ is a Priestley space, 
${\downarrow}y$ and ${\uparrow}x \cup \Max X$ are closed in~$X$. 
As ${\downarrow}y \cap ({\uparrow}x \cup \Max X) = \varnothing$, it follows that there exists a clopen down-set $V$ of $X$ such that $y \in V$, $x \notin V$ and $V \cap \Max X = \varnothing$. Since $\Max X = u(X)$, we can define the morphism $\varphi \colon \X \to \Z_0$ by
\begin{equation*}
\varphi(w) :=
\begin{cases}
 1, &\text{if $w \in X \comp V$,}\\
 a, &\text{if $w \in V$.}
 \end{cases}
\end{equation*}
We have $\varphi(x) = 1 \not\preccurlyeq a = \varphi(y)$, as required.

\Case{Case 2} {$y \in \Max X$.}%
Define $x_0 = u(x)$ and $y_0 = u(y)$ in $u(X) = \Max X$. As $x \not\preccurlyeq y$ and $y \in \Max X$, we have $x_0 \neq y_0$, by~(3) and (6).
Since $u^m(x_0) = x_0$, by~(6), we can choose the smallest number $k \ge 1$ such that $u^k(x_0) = x_0$, and $k$ must be a divisor of~$m$. We will use the substructure $\Z_k$ of $\ST_m$ from Lemma~\ref{lem:sepstructs}.

Let $V$ be a clopen subset of $\Max X$ such that $x_0 \in V$ and
\[
V \cap \{y_0, u(x_0), \dots, u^{k-1}(x_0)\} = \varnothing.
\]
Now let $\theta_V$ denote the equivalence relation on $\Max X$ with the two blocks $V$ and $\Max X \comp V$. Define
\[
\Syn {\theta_V} = \bigl\{\, (y, z) \in (\Max X)^2 \bigm| (\forall j \in \{0,1,\dotsc,m-1\})\ (u^j(y),u^j(z)) \in \theta_V \,\bigr\}.
\]
Using~(6), it is easy to check that
\begin{itemize}
\item $\Syn {\theta_V}$ is an equivalence relation on $\Max X$ such that each block is a clopen subset of $\Max X$,
\item $\Syn {\theta_V}$ is closed under $u$,
\item $\Syn {\theta_V}$ separates the elements $x_0$ and $y_0$, and
\item $\Syn {\theta_V}$ separates the elements $x_0, u(x_0), \dotsc, u^{k-1}(x_0)$.
\end{itemize}
In fact, the equivalence relation $\Syn {\theta_V}$ is the syntactic congruence of the unary algebra $\langle \Max X; u \rangle$ determined by~$\theta_V$ (see~\cite{CDFJ04}).

The closed substructure $\Max \X$ of $\X$ is an antichain. Thus we can now define the morphism $\psi \colon \Max \X \to \Z_k$ by
\begin{equation*}
\psi(w) :=
\begin{cases}
 a_i, &\text{if $(w,u^i(x_0)) \in \Syn {\theta_V}$, for $i \in \{0, 1, \dotsc, k-1\}$,}\\
 0, &\text{otherwise.}
 \end{cases}
\end{equation*}
Note that $u \colon \X \to \Max \X$ is a morphism, by~(2). 
Thus $\varphi : = \psi \circ u \colon \X \to \Z_k$ is a morphism satisfying $\varphi(x) = \psi(x_0) = a_0 \not\preccurlyeq \psi(y_0) = \varphi(y)$, as required.
\end{proof}

\begin{remark}\label{rem:shape}
We shall say that a structure $\X$ in $\IScP{\ST_m}$ is \defn{$u$-connected} if, for all $x,y \in X$, there exist $i,j$ with $u^i(x) = u^j(y)$. It follows straight from \ref{thm:topological}(2) that each structure in $\IScP{\ST_m}$ can be written as a disjoint union of $u$-connected substructures.

Now consider a $u$-connected structure~$\X$ in $\IScP{\ST_m}$. We want to show that $\X$ must have the general shape shown in Figure~\ref{fig:shape}. Let $x$ be a $\preccurlyeq$-maximal element of~$\X$. Then $x = u^m(x)$, by~\ref{thm:topological}(6). Choose the smallest number $k \ge 1$ such that $x = u^k(x)$. Then $k$ is a divisor of~$m$. For each $i \in \{0,\dotsc, k-1\}$, define $m_i := u^i(x)$ and let $P_i$ be the down-set of $\X$ generated by~$m_i$. Since $\X$ is $u$-connected, it follows from~\ref{thm:topological}(6) that $m_0,\dotsc,m_{k-1}$ are precisely the maximal elements of~$\X$. Using~\ref{thm:topological}(2), it follows that $u(P_i) = \{m_{i+1 \pmod k}\}$, for all $i \in \{0,\dotsc,k-1\}$. So the $\preccurlyeq$-connected components of $\X$ are $P_0, \dotsc, P_{k-1}$, and therefore $\X$ has the shape shown in Figure~\ref{fig:shape}.
\end{remark}

\begin{figure}[t]
\begin{tikzpicture}
      \clip (-0.5,-1.5) rectangle (7.5,1);
      \node[unshaded] (m0) at (0,0) {};
         \node[label, anchor=south] at (m0) {$m_0$};
      \coordinate (c0) at (0,-0.5) {};
        \node[label, anchor=north] at (c0) {$P_0$};
      \node[unshaded] (m1) at (1.5,0) {};
         \node[label, anchor=south] at (m1) {$m_1$};
      \coordinate (c1) at (1.5,-0.5) {};
        \node[label, anchor=north] at (c1) {$P_1$};
      \node[unshaded] (m2) at (3,0) {};
         \node[label, anchor=south] at (m2) {$m_2$};
      \coordinate (c2) at (3,-0.5) {};
        \node[label, anchor=north] at (c2) {$P_2$};
      \node (m3) at (4.5,0) {};
      \node at (4.75,-0.5) {$\cdots$};
      \coordinate (ck-2) at (5.5,-0.5) {};
      \node[unshaded] (mk-1) at (7,0) {};
         \node[label, anchor=south] at (mk-1) {$m_{k-1}$};
      \coordinate (ck-1) at (7,-0.5) {};
        \node[label, anchor=north] at (ck-1) {$P_{k-1}$};
      \draw[bigloop] (m0) to [out=-50,in=-130] (m0);
      \draw[bigloop] (m1) to [out=-50,in=-130] (m1);
      \draw[bigloop] (m2) to [out=-50,in=-130] (m2);
      \draw[bigloop] (mk-1) to [out=-50,in=-130] (mk-1);
      \draw[straight, densely dashed] (c0) to (m1);
      \draw[straight, densely dashed] (c1) to (m2);
      \draw[straight, densely dashed] (c2) to (m3);
      \draw[straight, densely dashed] (ck-2) to (mk-1);
      \draw[curvy, bend angle=35, densely dashed] (ck-1) to [bend right] (m0);
\end{tikzpicture}
\caption{The shape of a $u$-connected structure}
\label{fig:shape}
\end{figure}
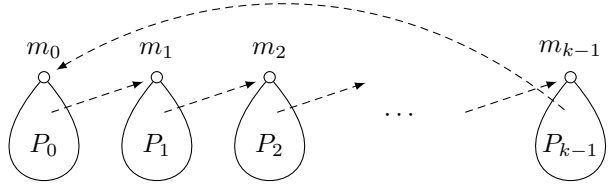

To show that two compatible relations on $\SB_m$ are equivalent, we use the following two general results from~\cite{DP10}. Compatible relations of the form $\homsetrel \X {S}$ were defined just before Lemma~\ref{lem:compat}.

\begin{lemma}[{\cite[2.6]{DP10}}]\label{lem:con1}
Let\/ $\A$ be a finite algebra. Let\/ $\homsetrel \X {S}$ and\/ $\homsetrel \Y {T}$ be compatible relations on~$\A$, associated with an alter ego~$\AT$ of\/~$\A$, such that\/ $T$ is a generating set for~$\Y$. Then\/ $\homsetrel \X {S}$ is conjunct-atomic definable from $\homsetrel \Y {T}$ if and only if the following
holds:
\begin{itemize}
\item
for each map $\phi \colon S \to A$ that does not extend to a morphism from $\X$ to~$\AT$, there exists a morphism\/ $\omega \colon \Y \to \X$ with $\omega(T) \subseteq S$ such that the map\/ $\phi \circ \omega\rest {T} \colon T \to A$ does not extend to a morphism from $\Y$ to~$\AT$.
\end{itemize}
\end{lemma}

\begin{lemma}[{\cite[5.1]{DP10}}]\label{lem:con2}
Let\/ $\A$ be a finite algebra. Let\/ $\homsetrel \X {S}$ and\/ $\homsetrel \Y {T}$ be compatible relations on~$\A$, associated with an alter ego~$\AT$ of\/~$\A$, such that\/ $S$ is a generating set for~$\X$. Assume that\/ $\Y \le \X$ and there is a retraction\/ $\rho \colon \X \twoheadrightarrow \Y$ with $\rho \rest {Y}=\id Y$ and $\rho(S)\subseteq T \subseteq S$. Then $\homsetrel \Y {T}$ is conjunct-atomic definable from $\homsetrel \X {S}$.
\end{lemma}

The next lemma restricts the number (up to equivalence) of compatible relations on $\SB_m$  that come from $u$-connected structures in $\IScP{\ST_m}$.

\begin{lemma}\label{lem:P(m1)}
Let\/ $m$ be odd, and let\/ $\X$ be a $u$-connected structure in\/ $\IScP{\ST_m}$ with generating set~$S$. Then the relation\/ $\homsetrel \X {S}$ is equivalent to\/ $\homsetrel \Y {S \cap Y}$, for some $\Y \le \X$ such that each $\preccurlyeq$-connected component of\/ $\Y$ has one of the following eight forms \textup(with elements of\/ $S$ shaded\textup).
\end{lemma}

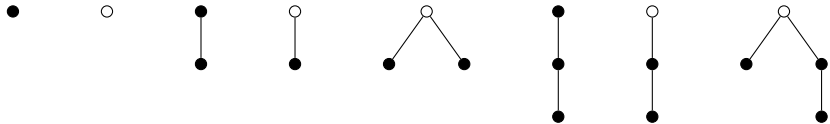
\smallskip
\begin{center}
\begin{tikzpicture}
     \node[shaded] (0) at (0,1.4) {};
     \node[unshaded] (1) at (1.25,1.4) {};
     \node[shaded] (2) at (2.5,1.4) {};
     \node[shaded] (3) at (2.5,0.7) {};
     \draw[order] (2) to (3);
     \node[unshaded] (4) at (3.75,1.4) {};
     \node[shaded] (5) at (3.75,0.7) {};
     \draw[order] (4) to (5);
     \node[unshaded] (9) at (5.5,1.4) {};
     \node[shaded] (10) at (5,0.7) {};
     \node[shaded] (11) at (6,0.7) {};
      \draw[order] (9) to (10);
      \draw[order] (9) to (11);
     \node[shaded] (6) at (7.25,1.4) {};
     \node[shaded] (7) at (7.25,0.7) {};
     \node[shaded] (8) at (7.25,0) {};
      \draw[order] (6) to (7);
      \draw[order] (7) to (8);
     \node[unshaded] (12) at (8.5,1.4) {};
     \node[shaded] (13) at (8.5,0.7) {};
     \node[shaded] (14) at (8.5,0) {};
      \draw[order] (12) to (13);
      \draw[order] (13) to (14);
     \node[unshaded] (15) at (10.25,1.4) {};
     \node[shaded] (16) at (9.75,0.7) {};
     \node[shaded] (17) at (10.75,0.7) {};
     \node[shaded] (18) at (10.75,0) {};
      \draw[order] (15) to (16);
      \draw[order] (15) to (17);
      \draw[order] (17) to (18);
\end{tikzpicture}
\end{center}

\begin{proof}
Assume $\X$ has a $\preccurlyeq$-connected component that does not have one of the eight allowable forms shown above. We will prove that there is a proper substructure $\Y$ of $\X$ such that the two relations $\homsetrel \X {S}$ and $\homsetrel \Y {S \cap Y}$ are equivalent, where $S \cap Y$ is a generating set for~$\Y$. Since $\X$ is finite, the result will follow by induction.

By Remark~\ref{rem:shape}, the structure $\X = \langle X; u, \preccurlyeq\rangle$ has the shape shown in Figure~\ref{fig:shape}, for some divisor $k$ of~$m$. Define $\Max X := \{m_0,m_1, \dotsc, m_{k-1}\}$. Then $X \comp \Max X \subseteq S$, as $u(X) = \Max X$ and the set $S$ generates~$\X$. In particular, we have $P_0\comp \{m_0\} \subseteq S$. Let $\P_0 = \langle P_0; \preccurlyeq\rangle$ be the induced ordered set on the $\preccurlyeq$-component $P_0$ of~$\X$. Without loss of generality, we can assume that $\P_0$ does not have one of the eight allowable forms. So one of the five cases described in Table~\ref{tab:proof} must apply.

Depending on which case applies, choose a subset $\{a\}$, $\{a,b\}$, $\{a,c\}$ or $\{a,b,c\}$ of $P_0$ according to the appropriate diagram of $\P_0$ in Table~\ref{tab:proof}. The sub-ordered set $\P_0'$ of $\P_0$ shown in Table~\ref{tab:proof} is properly contained in~$\P_0$, as we are assuming that $\P_0$ does not have one of the eight allowable forms.

\begin{table}[ht] 
\caption{Choosing the subset $P_0'$ of $P_0$}\label{tab:proof}
\begin{center}
\begin{tabular}{l l}
\hline
\parbox[t]{4.5cm}{${}$\\\emph{Case 1:} $\P_0$ has height $1$,\\ $m_0 \in S$, and $\abs {P_0} \ge 3$.}
&
\begin{tikzpicture}[baseline=1.5cm]
   \begin{scope}
     \node[anchor=west] at (-1.5,1) {$\P_0$};
     \node[shaded] (m0) at (0,1) {};
        \node[label, anchor=south] at (m0) {$m_0$};
     \node[shaded] (a1) at (-1,0) {};
        \node[label, anchor=north] at (a1) {$a$};
     \node[shaded] (a2) at (0,0) {};
     \node at (0.75,0) {$\cdots$};
     \node[shaded] (an) at (1.5,0) {};
      \draw[order] (m0) to (a1);
      \draw[order] (m0) to (a2);
      \draw[order] (m0) to (an);
   \end{scope}
   \begin{scope}[xshift=4cm]
     \node at (-1,1) {$\P_0'$};
     \node[shaded] (m0) at (0,1) {};
        \node[label, anchor=south] at (m0) {$m_0$};
     \node[shaded] (a) at (0,0) {};
        \node[label, anchor=north] at (a) {$a$};
      \draw[order] (m0) to (a);
   \end{scope}
\end{tikzpicture}
\\\hline
\parbox[t]{4.5cm}{${}$\\\emph{Case 2:} $\P_0$ has height $1$,\\ $m_0 \notin S$, and $\abs {P_0} \ge 4$.}
&
\begin{tikzpicture}[baseline=1.5cm]
   \begin{scope}
     \node[anchor=west] at (-1.5,1) {$\P_0$};
     \node[unshaded] (m0) at (0,1) {};
        \node[label, anchor=south] at (m0) {$m_0$};
     \node[shaded] (a1) at (-1,0) {};
        \node[label, anchor=north] at (a1) {$a$};
     \node[shaded] (a2) at (0,0) {};
        \node[label, anchor=north] at (a2) {$c$};
     \node at (0.75,0) {$\cdots$};
     \node[shaded] (an) at (1.5,0) {};
      \draw[order] (m0) to (a1);
      \draw[order] (m0) to (a2);
      \draw[order] (m0) to (an);
   \end{scope}
   \begin{scope}[xshift=4cm]
     \node at (-1,1) {$\P_0'$};
     \node[unshaded] (m0) at (0,1) {};
        \node[label, anchor=south] at (m0) {$m_0$};
     \node[shaded] (a) at (-0.5,0) {};
        \node[label, anchor=north] at (a) {$a$};
     \node[shaded] (c) at (0.5,0) {};
        \node[label, anchor=north] at (c) {$c$};
      \draw[order] (m0) to (a);
      \draw[order] (m0) to (c);
   \end{scope}
\end{tikzpicture}
\\\hline
\parbox[t]{4.5cm}{${}$\\\emph{Case 3:} $\P_0$ has height at\\ least $2$, and $m_0 \in S$.}
&
\begin{tikzpicture}[baseline=2cm]
   \clip (-1.5,-0.6) rectangle (5,2);
   \begin{scope}
      \node[anchor=west] at (-1.5,1.5) {$\P_0$};
      \node[shaded] (m0) at (0,1.5) {};
         \node[label, anchor=south] at (m0) {$m_0$};
      \node[shaded] (a) at (0,0.75) {};
         \node[label, anchor=west] at (a) {$b$};
      \node[shaded] (b) at (0,0) {};
         \node[label, anchor=west] at (b) {$a$};
      \draw[densely dotted, min distance=100pt] (m0) to [out=-135,in=-45] (m0);
      \draw[order] (m0) to (a);
      \draw[order] (a) to (b);
   \end{scope}
   \begin{scope}[xshift=4cm]
      \node at (-1,1.5) {$\P_0'$};
     \node[shaded] (m0) at (0,1.5) {};
        \node[label, anchor=south] at (m0) {$m_0$};
     \node[shaded] (a) at (0,0.75) {};
        \node[label, anchor=west] at (a) {$b$};
     \node[shaded] (b) at (0,0) {};
        \node[label, anchor=west] at (b) {$a$};
      \draw[order] (m0) to (a);
      \draw[order] (a) to (b);
   \end{scope}
\end{tikzpicture}
\\\hline
\parbox[t]{4.5cm}{${}$\\\emph{Case 4:} $\P_0$ has height at\\ least $2$, $m_0 \notin S$, and $P_0\comp
\{m_0\}$\\ is order-connected.}
&
\begin{tikzpicture}[baseline=2cm]
   \clip (-1.5,-0.6) rectangle (5,2);
   \begin{scope}
      \node[anchor=west] at (-1.5,1.5) {$\P_0$};
      \node[unshaded] (m0) at (0,1.5) {};
         \node[label, anchor=south] at (m0) {$m_0$};
      \node[shaded] (a) at (0,0.75) {};
         \node[label, anchor=west] at (a) {$b$};
      \node[shaded] (b) at (0,0) {};
         \node[label, anchor=west] at (b) {$a$};
      \draw[densely dotted, min distance=100pt] (m0) to [out=-135,in=-45] (m0);
      \draw[order] (m0) to (a);
      \draw[order] (a) to (b);
   \end{scope}
   \begin{scope}[xshift=4cm]
      \node at (-1,1.5) {$\P_0'$};
     \node[unshaded] (m0) at (0,1.5) {};
        \node[label, anchor=south] at (m0) {$m_0$};
     \node[shaded] (a) at (0,0.75) {};
        \node[label, anchor=west] at (a) {$b$};
     \node[shaded] (b) at (0,0) {};
        \node[label, anchor=west] at (b) {$a$};
      \draw[order] (m0) to (a);
      \draw[order] (a) to (b);
   \end{scope}
\end{tikzpicture}
\\\hline
\parbox[t]{4.5cm}{${}$\\\emph{Case 5:} $\P_0$ has height at\\ least $2$, $m_0 \notin S$, and $P_0\comp
\{m_0\}$\\ is order-disconnected.}
&
\begin{tikzpicture}[baseline=2cm]
  \clip (-1.5,-0.65) rectangle (5,2);
  \begin{scope}
   \node[anchor=west] at (-1.5,1.5) {$\P_0$};
     \node[unshaded] (m0) at (0,1.5) {};
        \node[label, anchor=south] at (m0) {$m_0$};
     \node[shaded] (a) at (-0.5,0.75) {};
        \node[label, anchor=east] at (a) {$b$};
     \node[shaded] (b) at (-0.5,0) {};
        \node[label, anchor=east] at (b) {$a$};
     \node[shaded] (c) at (0.5,0.75) {};
        \node[label, anchor=west] at (c) {$c$};
     \draw[densely dotted, min distance=100pt] (m0) to [out=-25,in=-90] (m0);
     \draw[densely dotted, min distance=100pt] (m0) to [out=-155,in=-90] (m0);
      \draw[order] (m0) to (a);
      \draw[order] (m0) to (c);
      \draw[order] (a) to (b);
   \end{scope}
 \begin{scope}[xshift=4cm]
   \node at (-1,1.5) {$\P_0'$};
     \node[unshaded] (m0) at (0,1.5) {};
        \node[label, anchor=south] at (m0) {$m_0$};
     \node[shaded] (a) at (-0.5,0.75) {};
        \node[label, anchor=east] at (a) {$b$};
     \node[shaded] (b) at (-0.5,0) {};
        \node[label, anchor=east] at (b) {$a$};
     \node[shaded] (c) at (0.5,0.75) {};
        \node[label, anchor=west] at (c) {$c$};
      \draw[order] (m0) to (a);
      \draw[order] (m0) to (c);
      \draw[order] (a) to (b);
   \end{scope}
\end{tikzpicture}
\\\hline
\end{tabular}
\end{center}
\end{table}

Now define the proper substructure $\Y$ of $\X$ by $Y := P_0' \cup (X \comp P_0)$  and define $T:= Y \cap S$. Then $T$ is a generating set for~$\Y$. We shall prove that the two compatible relations $\homsetrel \X {S}$ and $\homsetrel \Y {T}$ on $\SB_m$ are equivalent.

\Claim{Claim 1} {$\homsetrel \Y {T}$ is conjunct-atomic definable from $\homsetrel \X {S}$.}%
We will use Lemma~\ref{lem:con2}. In each of the five cases in Table~\ref{tab:proof}, it is easy to find an order-preserving map $\rho_0 \colon \P_0 \to \P_0'$ such that $\rho_0 \rest {P_0'} = \id {P_0'}$ and $\rho_0^{-1}(m_0) = \{m_0\}$. We can then define $\rho \colon \X \twoheadrightarrow \Y$ by $\rho = \rho_0 \cup \id {X \comp P_0}$, with $\rho \rest {Y} = \id Y$ and $\rho(S) \subseteq T$. Hence $\homsetrel \Y {T}$ is conjunct-atomic definable from $\homsetrel \X {S}$, by Lemma~\ref{lem:con2}.

\Claim{Claim 2} {$\homsetrel \X {S}$ is conjunct-atomic definable from $\homsetrel \Y {T}$.}%
We will use Lemma~\ref{lem:con1}. Let $\varphi \colon S \to S_m$ such that $\varphi$ does not extend to a morphism from $\X$ to~$\ST_m$. We begin by checking that one of the following four conditions holds:
\begin{enumerate}[ (a)]  
\item  
$\varphi \rest {T}$ does not extend to a morphism from $\Y$
    to~$\ST_m$;
\item  
there are $x_1,x_2 \in P_0 \comp \{m_0\}$ with $x_1 \preccurlyeq x_2$
    but $\varphi(x_1) \not\preccurlyeq \varphi(x_2)$;
\item  
$m_0 \in S$ and there is $x \in P_0 \comp \{m_0\}$ with
    $\varphi(x) \not\preccurlyeq \varphi(m_0)$;
\item  
$m_0 \notin S$ and there are $x_1,x_2 \in P_0 \comp \{m_0\}$ such
    that $\varphi(x_1)$ and $\varphi(x_2)$ belong to different
    $\preccurlyeq$-components of~$\ST_m$.
\end{enumerate}

To see that one of these conditions holds, assume that (a) fails. Then $\varphi \rest {T}$ extends to a morphism $\psi \colon \Y \to \ST_m$. Note that $X \comp S \subseteq \Max X \subseteq Y$. So $X = Y \cup S$. Since $Y \cap S = T$, we can define the map $\chi \colon X \to S_m$ by $\chi = \psi \cup \varphi$.

First suppose that $\chi$ is $\preccurlyeq$-preserving. By Theorem~\ref{thm:topological}, both $\X$ and $\ST_m$ satisfy $x \preccurlyeq y \implies u(x) \approx u(y)$. Since $u(X) = \Max X \subseteq Y$ and $\psi \colon \Y \to \ST_m$ is a morphism, it follows that $\chi$ preserves~$u$, and therefore $\chi \colon \X \to \ST_m$ is a morphism. Since $\chi$ extends~$\phi$, it cannot be a morphism from $\X$ to~$\ST_m$. Thus we have shown that $\chi$ is not $\preccurlyeq$-preserving.

Since $X \comp P_0 \subseteq Y$ and $\psi \colon \Y \to \ST_m$ is a morphism, it now follows that $\chi\rest{P_0}$ is not $\preccurlyeq$-preserving.
We can assume that (b) fails, and hence $\varphi\rest{P_0\comp\{m_0\}}$ is $\preccurlyeq$-preserving. Consequently, if $m_0\in S$, then (c) must hold. Now assume that $m_0\notin S$. We have $m_0,a \in P_0' \subseteq Y$ with $\psi(a) \preccurlyeq \psi(m_0)$. Since $\ST_m$ satisfies~\ref{thm:topological}(6), we know that $\psi(m_0)$ is a maximal element of~$\ST_m$. By \ref{thm:topological}(4), each $\preccurlyeq$-component of~$\ST_m$ has a greatest element. So if $\varphi(P_0\comp\{m_0\})$ were contained in a single $\preccurlyeq$-component of~$\ST_m$, then as $\varphi(a) = \psi(a) \preccurlyeq \psi(m_0)$ and $a\in P_0\comp\{m_0\}$, we would have $\phi(P_0\comp\{m_0\}) \subseteq {\downarrow}\psi(m_0)$, in which case $\chi\rest{P_0}$ would be $\preccurlyeq$-preserving, a contradiction. Hence (d) holds.

We have shown that one of the conditions (a)--(d) holds. In each of these four cases, we will find a morphism  $\omega \colon \Y \to \X$ with $\omega(T) \subseteq S$ such that the map $\phi \circ \omega\rest {T} \colon T \to S_m$ does not extend to a morphism from $\Y$ to~$\ST_m$. It will then follow by Lemma~\ref{lem:con1} that $\homsetrel \X {S}$ is conjunct-atomic definable from $\homsetrel \Y {T}$, as required.

\Case{Case a} {$\varphi \rest {T}$ does not extend to a morphism from $\Y$ to $\ST_m$.}%
Take $\omega \colon \Y \to \X$ to
be the inclusion. Then $\omega(T) = T \subseteq S$ and the map $\phi \circ \omega\rest {T}  = \phi \rest {T}$ does not extend to a morphism from $\Y$ to~$\ST_m$, by assumption.

\Case{Case b} {there are $x_1,x_2 \in P_0 \comp \{m_0\}$ with $x_1 \preccurlyeq x_2$ but $\varphi(x_1) \not\preccurlyeq \varphi(x_2)$.}%
Only Cases 3, 4 and 5 from Table~\ref{tab:proof} can apply. So we can define the morphism $\omega \colon \Y \to \X$ by
\begin{equation*}
 \omega(y) =
 \begin{cases}
 x_2, &\text{if $y = b$,}\\
 x_1, &\text{if $y = a$,}\\
 y, &\text{otherwise.}
 \end{cases}
\end{equation*}
Then $\omega(T) \subseteq S$. The map $\phi \circ \omega \rest {T} \colon T \to S_m$ does not extend to a morphism from $\Y$ to~$\ST_m$, because $a \preccurlyeq b$ in $\Y$ but $\varphi \circ \omega (a) = \varphi(x_1) \not\preccurlyeq \varphi(x_2) = \varphi \circ \omega (b)$ in~$\ST_m$.

\Case{Case c} {$m_0 \in S$ and there is $x \in P_0 \comp \{m_0\}$ with $\varphi(x) \not\preccurlyeq \varphi(m_0)$.}%
Only Cases~1 and~3 from Table~\ref{tab:proof} can apply. Define the morphism $\omega_1 \colon \Y \to \X$ in Case~1 and the morphism $\omega_3 \colon \Y \to \X$ in Case~3 by
\begin{equation*}
 \omega_1(y) =
 \begin{cases}
 x, &\text{if $y = a$,}\\
 y, &\text{otherwise,}
 \end{cases}
\quad\text{and}\quad
 \omega_3(y) =
 \begin{cases}
 x, &\text{if $y \in \{a,b\}$,}\\
 y, &\text{otherwise.}
 \end{cases}
\end{equation*}
Then $\omega_i(T) \subseteq S$. The map $\phi \circ \omega_i \rest {T} \colon T \to S_m$ does not extend to a morphism from $\Y$ to~$\ST_m$, because $a \preccurlyeq m_0$ in $\Y$ but $\varphi \circ \omega_i (a) = \varphi(x) \not\preccurlyeq \varphi(m_0) = \varphi \circ \omega_i (m_0)$ in~$\ST_m$.

\Case{Case d} {$m_0 \notin S$ and there are $x_1,x_2 \in P_0 \comp \{m_0\}$ such that $\varphi(x_1)$ and $\varphi(x_2)$ belong to different $\preccurlyeq$-components of\/~$\ST_m$.}%
By Case~b, we can assume that $\varphi$ is $\preccurlyeq$-preserving on $P_0 \comp \{m_0\}$. So only Cases~2 and~5 from Table~\ref{tab:proof} can apply. Define $\omega_2 \colon \Y \to \X$ in Case~2 and $\omega_5 \colon \Y \to \X$ in Case~5 by
\begin{equation*}
 \omega_2(y) =
 \begin{cases}
 x_1, &\text{if $y = a$,}\\
 x_2, &\text{if $y = c$,}\\
 y, &\text{otherwise,}
 \end{cases}
\quad\text{and}\quad
 \omega_5(y) =
 \begin{cases}
 x_1, &\text{if $y \in \{a,b\}$,}\\
 x_2, &\text{if $y = c$,}\\
 y, &\text{otherwise.}
 \end{cases}
\end{equation*}
Then $\varphi \circ \omega_i \rest {T} \colon T \to S_m$ does not extend to a morphism from $\Y$ to~$\ST_m$, because $a$ and $c$ belong to the same $\preccurlyeq$-component $P_0'$ of~$\Y$, but $\varphi \circ \omega_i(a) = \varphi(x_1)$ and $\varphi \circ \omega_i(c) = \varphi(x_2)$ belong to different $\preccurlyeq$-components of~$\ST_m$.
\end{proof}

\begin{lemma}
For each odd\/~$m$, the Ockham algebra $\SB_m$ with dual space $\DT_m$ admits only finitely many relations.
\end{lemma}

\begin{proof}
The alternating alter ego $\ST_m$ strongly dualises~$\SB_m$, by Theorem~\ref{thm:duality}, and so (IC) holds. Therefore, by Lemma~\ref{lem:compat}, every compatible relation on $\SB_m$ is equivalent to one of the form $\homsetrel \X {S}$, for some structure $\X \in \IScP {\ST_m}$ and some generating set $S$ for~$\X$.

By Remark~\ref{rem:shape}, such a structure $\X$ is the disjoint union of its $u$-connected substructures. Assume that $\X$ is not $u$-connected. We can write $\X = \X_1 \dotcup \X_2$, where $\X_1$ and $\X_2$ are non-empty substructures of~$\X$. Then $S_1 := X_1 \cap S$ is a generating set for~$\X_1$, and $S_2 := X_2 \cap S$ is a generating set for~$\X_2$. We have
\begin{align*}
\homsetrel \X {S} &=  \{\, \alpha \rest S \mid \alpha \colon \X_1 \dotcup \X_2 \to \ST_m \,\}\\
&\equiv \{\, \alpha_1 \rest {S_1} \mid \alpha_1 \colon \X_1 \to \ST_m \,\} \times \{\, \alpha_2 \rest {S_2} \mid \alpha_2 \colon \X_2 \to \ST_m \,\}\\
 &= \homsetrel {\X_1} {S_1} \times \homsetrel {\X_2} {S_2}.
\end{align*}
Since the structures $\X_1$ and $\X_2$ are non-empty, we can use the substructure $\Z_1$ of~$\ST_m$ given by Lemma~\ref{lem:sepstructs} to see that the relations $\homsetrel {\X_1} {S_1}$ and $\homsetrel {\X_2} {S_2}$ are non-trivial. So the relation $\homsetrel \X {S}$ is not directly indecomposable.

Using Lemma~\ref{lem:indec}, we now only need to find a finite upper bound on the number (up to equivalence) of relations $\homsetrel \X {S}$ such that $\X$ is a $u$-connected structure in $\IScP{\ST_m}$ and $S$ is a generating set for~$\X$. Each $u$-connected structure in $\IScP{\ST_m}$ has at most $m$ $\preccurlyeq$-connected components, by Remark~\ref{rem:shape}. So, by Lemma~\ref{lem:P(m1)}, we can use the upper bound $m \times 8^m$.
\end{proof}

It follows from the previous lemma with $m = 1$ that the $3$-element Stone algebra $\SB_1$ admits only finitely many relations. This was claimed without proof in~\cite{DP10}.

\section{Ockham algebras with infinitely many relations}\label{sec:inf}

In this section, we shall check that the eight finite Ockham algebras whose dual spaces are given in Figure~\ref{fig:infiniteOckham} each admit infinitely many relations. Using symmetry, there are only six algebras to consider. Two of these algebras are already known to admit infinitely many relations for general reasons.

\begin{lemma}[{\cite[3.4]{DP10}}]\label{lem:DP}\
\begin{enumerate}[ \normalfont(1)] 
\item
The dual of the Ockham space $\Y_1$ is the $4$-element Boolean algebra, which admits infinitely many relations because it is the square of a non-trivial algebra.
\item
The dual of the Ockham space $\Y_4$ is the Stone algebra on the $4$-element chain, which admits infinitely many relations because it has a pair of non-permuting congruences.
\end{enumerate}
\end{lemma}

So it remains to consider the Ockham algebras corresponding to $\Y_2$, $\Y_3$, $\Y_5$ and~$\Y_6$. We will be able to deal with these four algebras two at a time. To show that an algebra admits infinitely many relations we will adapt the following technique from~\cite{DP10}.

\begin{lemma}[{\cite[2.7]{DP10}}]\label{lem:inf}
Let\/ $\A$ be a finite algebra. To show that\/ $\A$ admits infinitely many relations, it suffices to find
\begin{itemize}
\item
an alter ego $\AT$ of\/~$\A$, and
\item
for each $n\ge 1$, a structure $\X_n \in \IScP\AT$ and a map\/ $\phi_n \colon X_n \to A$ that is not a morphism from\/ $\X_n$ to~$\AT$
\end{itemize}
such that the following holds, either for all\/ $k < \ell$ or for all\/ $k > \ell$:
\begin{itemize}
\item
for each morphism $\omega \colon \X_k \to \X_\ell$, the map $\phi_\ell \circ \omega \colon X_k \to A$ is a morphism from $\X_k$ to~$\AT$.
\end{itemize}
\end{lemma}

In the proof of this lemma, the sequence of structures $\X_1,\X_2,\X_3,\dotsc$ is used to define a sequence of compatible relations $r_1, r_2, r_3, \dotsc$
on~$\A$, where $r_n := \homsetrel {\X_n}{X_n}$. The assumptions of the lemma are set up to ensure that these relations are pairwise non-equivalent.

We start with the Ockham space~$\Y_3$, which is dual to the $3$-element Kleene algebra $\KB = \langle \{0,a,1\}; \vee, \wedge, f, 0, 1\rangle$ shown in Figure~\ref{fig:K}. Define the enriched ordered set $\KT := \langle \{0,a,1\}; \preccurlyeq, K_0\rangle$ shown in Figure~\ref{fig:K}, where $K_0 = \{0,1\}$. It is easy to check that $\KT$ is an alter ego of\/~$\KB$. (In fact, the two relations $\preccurlyeq$ and $K_0$ determine the clone of~$\KB$; see~\cite[4.3.12]{CD98}.) Rather than applying Lemma~\ref{lem:inf} directly to $\KB$ and~$\KT$, we prove a more general result that will also cover the Ockham algebra with dual space~$\Y_2$.

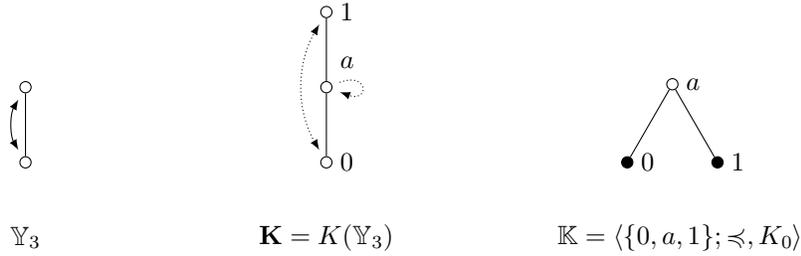
\begin{figure}[t]
\begin{tikzpicture}
   \begin{scope}
     \node at (0,-1) {$\Y_3$};
     \node[unshaded] (0) at (0,0) {};
     \node[unshaded] (1) at (0,1) {};
     \draw[order] (0) to (1);
     \draw[curvy, <->] (0) to [bend left] (1);
     \end{scope}
   \begin{scope}[xshift=4cm]
      \node at (0,-1) {$\KB = K(\Y_3)$};
      \node[unshaded] (0) at (0,0) {};
         \node[label, anchor=west] at (0) {$0$};
      \node[unshaded] (a) at (0,1) {};
         \node[label, anchor=south west] at (a) {$a$};
      \node[unshaded] (1) at (0,2) {};
         \node[label, anchor=west] at (1) {$1$};
      \draw[order] (0) to (a);
      \draw[order] (a) to (1);
      \draw[curvy, <->, densely dotted] (0) to [bend left](1);
      \draw[loopy, densely dotted] (a) to [out=20,in=-20] (a);
         \end{scope}
   \begin{scope}[xshift=8cm]
      \node at (0.7,-1) {$\KT = \langle \{0,a,1\}; \preccurlyeq, K_0\rangle$};
      \node[shaded] (0) at (0,0) {};
         \node[label, anchor=west] at (0) {$0$};
      \node[unshaded] (a) at ($(0)+(60:1.2)$) {};
         \node[label, anchor=west] at (a) {$a$};
      \node[shaded] (1) at ($(a)+(-60:1.2)$) {};
         \node[label, anchor=west] at (1) {$1$};
      \draw[order] (0) to (a);
      \draw[order] (1) to (a);
    \end{scope}
\end{tikzpicture}
\caption{The $3$-element Kleene algebra $\KB$}\label{fig:K}
\end{figure}

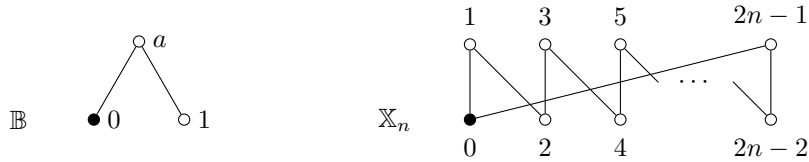
\begin{figure}[t]
\begin{tikzpicture}
   \begin{scope}
      \node at (-1,0) {$\BT$};
      \node[shaded] (0) at (0,0) {};
         \node[label, anchor=west] at (0) {$0$};
      \node[unshaded] (a) at ($(0)+(60:1.2)$) {};
         \node[label, anchor=west] at (a) {$a$};
      \node[unshaded] (1) at ($(a)+(-60:1.2)$) {};
         \node[label, anchor=west] at (1) {$1$};
      \draw[order] (0) to (a);
      \draw[order] (1) to (a);
 \end{scope}
 \begin{scope}[xshift=5cm]
   \node at (-1,0) {$\X_n$};
      \node[shaded] (0) at (0,0) {};
         \node[label, anchor=north] at (0) {$0$};
      \node[unshaded] (1) at (0,1) {};
         \node[label, anchor=south] at (1) {$1$};
      \node[unshaded] (2) at (1,0) {};
         \node[label, anchor=north] at (2) {$2$};
      \node[unshaded] (3) at (1,1) {};
         \node[label, anchor=south] at (3) {$3$};
       \node[unshaded] (4) at (2,0) {};
         \node[label, anchor=north] at (4) {$4$};
      \node[unshaded] (5) at (2,1) {};
         \node[label, anchor=south] at (5) {$5$};
      \coordinate (6) at (3,0) {};
      \node at (3,0.5) {$\cdots$};
      \coordinate (2n-3) at (3,1) {};
      \node[unshaded] (2n-2) at (4,0) {};
         \node[label, anchor=north] at (2n-2) {$2n-2$};
      \node[unshaded] (2n-1) at (4,1) {};
         \node[label, anchor=south] at (2n-1) {$2n-1$};
      \draw[order] (0) to (1);
       \draw[order] (0) to (2n-1);
      \draw[order] (2) to (1);
      \draw[order] (2) to (3);
      \draw[order] (4) to (3);
      \draw[order] (4) to (5);
      \draw[order, shorten >=0.7cm] (5) to (6);
      \draw[order, shorten <=0.7cm] (2n-3) to (2n-2);
      \draw[order] (2n-2) to (2n-1);
 \end{scope}
\end{tikzpicture}
\caption{Structures for Lemma~\ref{lem:gen1}}\label{fig:gen1}
\end{figure}

\begin{lemma}\label{lem:gen1}
Let\/ $\A$ be a finite algebra, let\/ $\AT = \langle A; r, s\rangle$ be an alter ego of\/~$\A$, where $r \subseteq A^2$ and $s \subseteq A$, and let\/ $\BT = \langle \{0,a,1\}; \le, \{0\} \rangle$ be the enriched ordered set shown in Figure~\ref{fig:gen1}. If\/ $\IScP \AT$ contains the structure\/ $\BT$, then\/ $\A$ admits infinitely many relations.

\end{lemma}

\begin{proof}
Assume that\/ $\IScP \AT$ contains the structure $\BT$.
We shall use Lemma~\ref{lem:inf} to show that $\A$ admits infinitely many relations. Let $n \ge 2$ and define the structure $\X_n = \langle \{0,1,\dots, 2n-1\}; \le, s\rangle$ as in Figure~\ref{fig:gen1}: the ordered set $\langle X_n; \le \rangle$ is the $2n$-element crown and $s = \{0\}$. Now define the map $\psi_n \colon X_n \to B$ by
\begin{equation*}
 \psi_n(i) =
 \begin{cases}
 1, &\text{if $i=n$,}\\
 0, &\text{otherwise.}
 \end{cases}
\end{equation*}
Then $\psi_n$ is not a morphism from $\X_n$ to~$\BT$, as $n$ and $n+1$ are comparable in $\X_n$, but $\psi_n(n) = 1$ and $\psi_n(n+1)= 0$ are not comparable in~$\BT$. Since $\BT \in \IScP \AT$ by assumption, there must be a morphism $\rho_n \colon \BT \to \AT$ such that $\phi_n := \rho_n \circ \psi_n$ is not a morphism from $\X_n$ to~$\AT$.

Using Lemma~\ref{lem:inf}, the following two claims establish that $\A$ admits infinitely many relations.

\Claim{Claim 1} {$\X_n \in \IScP\AT$, for all $n \ge 2$.}%
Since $\BT \in \IScP \AT$, it is enough to show that $\X_n \in \IScP \BT$. First, let $x,y \in X_n$ with $x \nle y$ in $\X_n$. If $x \ne 0$, then we can define the morphism $\alpha_x \colon \X_n \to \BT$ by
\begin{equation*}
 \alpha_x(z) =
 \begin{cases}
 a, &\text{if $z \in {\uparrow}x$,}\\
 0, &\text{otherwise,}
 \end{cases}
\end{equation*}
and we have $\alpha_x(x) = a \nle 0 = \alpha_x(y)$ in $\BT$. If $x = 0$, then we can define the morphism $\beta \colon \X_n \to \BT$ by
\begin{equation*}
 \beta(z) =
 \begin{cases}
 0, &\text{if $z=0$,}\\
 a, &\text{if $z=1$ or $z=2n-1$,}\\
 1, &\text{otherwise,}
 \end{cases}
\end{equation*}
and we have $\beta(x) = 0 \nle 1 = \beta(y)$ in $\BT$. Thus the order $\le$ is separated by morphisms from $\X_n$ to~$\BT$, and it follows that the elements of $X_n$ are also separated. Finally, define the morphism $\gamma \colon \X_n \to \BT$ by
\begin{equation*}
 \gamma(z) =
 \begin{cases}
 0, &\text{if $z=0$,}\\
 a, &\text{otherwise.}
 \end{cases}
\end{equation*}
Then $\gamma(X_n \comp \{0\}) \subseteq B \comp \{0\}$, and so $\gamma$ separates the unary relation~$s$. Hence we have shown that $\X_n \in \IScP \BT$.

\Claim{Claim 2} {Let\/ $\omega \colon \X_k \to \X_\ell$, where $2 \le k < \ell$. Then $\phi_\ell \circ \omega$ is a morphism from $\X_k$ to~$\AT$.}%
Since $\phi_\ell := \rho_\ell \circ \psi_\ell$, it suffices to show that $\psi_\ell \circ \omega$ is a morphism from $\X_k$ to~$\BT$.

For any connected ordered set~$\P$, there is a natural distance function $d$ on~$\P$, where $d(a,b)$ is the length of the shortest fence in $\P$ between $a$ and~$b$. We will use the distance functions $d_k$ and $d_\ell$ on $\X_k$ and $\X_\ell$.

We first show that $\omega(X_k) \subseteq X_\ell \comp \{\ell\}$. Let $x \in X_k$. The $2k$-crown $\X_k$ has diameter~$k$, and so $d_k(0, x) \le k$. Note that $\omega(0)=0$, as $\omega$ preserves~$s$. Since $\omega$ is order-preserving, we have
\[
d_\ell(0, \omega(x)) = d_\ell(\omega(0), \omega(x)) \le d_k(0, x) \le k.
\]
But $d_\ell(0, \ell) = \ell > k$, and therefore $\omega(x) \in  X_\ell \comp \{\ell\}$.

The restriction of the map $\psi_\ell$ to $X_\ell \comp \{\ell\}$ is constant~$0$, and so preserves both $\le$ and~$s$. 
Since $\omega(X_k) \subseteq X_\ell \comp \{\ell\}$, it now follows that $\psi_\ell \circ \omega \colon \X_k \to \BT$ is a morphism.
\end{proof}

\begin{lemma}\label{lem:K}
The $3$-element Kleene algebra~$\KB$ \textup(which is the Ockham algebra with dual space~$\Y_3$\textup) admits infinitely many relations.
\end{lemma}
\begin{proof}
This follows from Lemma~\ref{lem:gen1} using the alter ego $\KT = \langle \{0,a,1\}; \preccurlyeq, K_0\rangle$ of $\KB$ shown in Figure~\ref{fig:K}. The structure $\BT$ from Figure~\ref{fig:gen1} embeds into $\KT^2$ via $0 \mapsto (0,0)$, $a \mapsto (a,a)$, $1 \mapsto (1,a)$.
\end{proof}

\begin{lemma}\label{lem:a2}
The Ockham algebra with dual space $\Y_2$ admits infinitely many relations.
\end{lemma}

\begin{proof}
The dual of $\Y_2$ is the Ockham algebra $\A_2 = \langle \{0, a, b, 1\}; \vee, \wedge, f, 0, 1 \rangle$ shown in Figure~\ref{fig:dualY2}. Define the enriched ordered set $\AT_2 = \langle \{0,a,b,1\}; \preccurlyeq, s \rangle$ as in Figure~\ref{fig:dualY2}, where $s = \{0,1\}$. Then $\AT_2$ is an alter ego of~$\A_2$. The structure $\BT$ from Lemma~\ref{lem:gen1} embeds into $(\AT_2)^2$ via $0 \mapsto (0,1)$, $a \mapsto (a,1)$, $1 \mapsto (a,b)$. So $\A_2$ admits infinitely many relations.
\end{proof}

\begin{figure}[ht]
\begin{tikzpicture}
   \begin{scope}
      \node at (0.5,-1) {$\Y_2$};
      \node[unshaded] (0) at (0,0) {};
      \node[unshaded] (1) at (1,0) {};
      \draw[curvy] (0) to [bend left] (1);
      \draw[loopy] (1) to [out=110,in=70] (1);
   \end{scope}
   \begin{scope}[xshift=4cm]
   \node at (0,-1) {$\A_2 = K(\Y_2)$};
      \node[unshaded] (0) at (0,0) {};
         \node[label, anchor=north] at (0) {$0$};
      \node[unshaded] (a) at ($(0) + (135:1)$) {};
         \node[label, anchor=east] at (a) {$a$};
      \node[unshaded] (b) at ($(0) + (45:1)$) {};
         \node[label, anchor=west] at (b) {$b$};
      \node[unshaded] (1) at ($(a) + (45:1)$) {};
         \node[label, anchor=south] at (1) {$1$};
      \draw[order] (0) to (a);
      \draw[order] (0) to (b);
      \draw[order] (a) to (1);
      \draw[order] (b) to (1);
      \draw[straight, <->, densely dotted] (0) to (1);
      \draw[curvy, densely dotted] (a) to [bend left] (1);
      \draw[curvy, densely dotted] (b) to [bend left] (0);
         \end{scope}
   \begin{scope}[xshift=8cm]
      \node at (0.6,-1) {$\AT_2 = \langle A_2; \preccurlyeq, \{0,1\}\rangle$};
       \node[shaded] (0) at (0,0) {};
         \node[label, anchor=west] at (0) {$0$};
      \node[unshaded] (a) at (0,1) {};
         \node[label, anchor=west] at (a) {$a$};
     \node[unshaded] (b) at (1,0) {};
         \node[label, anchor=west] at (b) {$b$};
     \node[shaded] (1) at (1,1) {};
         \node[label, anchor=west] at (1) {$1$};
      \draw[order] (a) to (0);
      \draw[order] (1) to (b);
    \end{scope}
\end{tikzpicture}
\caption{The Ockham algebra with dual space $\Y_2$}\label{fig:dualY2}
\end{figure}
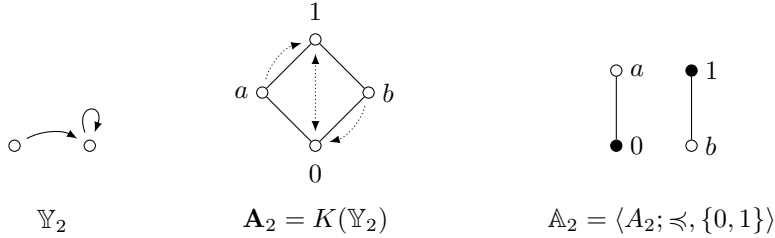

Our second general lemma will cover the remaining two Ockham algebras $K(\Y_5)$ and $K(\Y_6)$.

\begin{figure}[t]
\begin{tikzpicture}
   \begin{scope}
      \node at (-1,0) {$\BT$};
       \node[unshaded] (0) at (0,0) {};
         \node[label, anchor=west] at (0) {$0$};
      \node[unshaded] (a) at (0,1.1) {};
         \node[label, anchor=west] at (a) {$a$};
     \node[unshaded] (1) at (1.1,0) {};
         \node[label, anchor=west] at (1) {$1$};
     \node[unshaded] (b) at (1.1,1.1) {};
         \node[label, anchor=west] at (b) {$b$};
      \draw[order] (a) to (0);
      \draw[order] (1) to (b);
  \begin{pgfonlayer}{background}
    \node[ellipse,draw,dashed,fit=(0)(a),inner xsep=10pt,inner ysep=2pt] {};
    \node[circle,draw,dashed,fit=(b),inner sep=7pt] {};
    \node[circle,draw,dashed,fit=(1),inner sep=7pt] {};
  \end{pgfonlayer}
 \end{scope}
 \begin{scope}[xshift=5cm, yshift=0.25cm]
    \node at (-1.25,-0.25) {$\X_n$};
    \node at (5,-0.25) {$L_n$};
    \node at (5,1.5) {$U_n$};
      \node[unshaded] (xn) at (0,-0.25) {};
         \node[label, anchor=west] at (xn) {$0$};
      \node[unshaded] (1) at (0.5,1.75) {};
         \node[label] at (0.5,2.05) {$1$};
      \node[unshaded] (2) at (1,1) {};
         \node[label] at (1,0.7) {$2$}; 
      \node[unshaded] (3) at (1.5,1.75) {};
         \node[label] at (1.5,2.05) {$3$}; 
      \coordinate (4) at (2,1);
       \node at (2.25,1.375) {$\cdots$};
      \coordinate (2n-3) at (2.5, 1.75);
      \node[unshaded] (2n-2) at (3,1) {};
         \node[label] at (3,0.7) {$2n-2$}; 
      \node[unshaded] (2n-1) at (3.5,1.75) {}; 
         \node[label] at (3.5,2.05) {$2n-1$}; 
      \node[unshaded] (yn) at (4,-0.25) {};
         \node[label, anchor=east] at (yn) {$2n$};
       \draw[order] (xn) to (1);
       \draw[order] (2) to (1);
       \draw[order] (2) to (3);
       \draw[order, shorten >=15pt] (3) to (4);
       \draw[order, shorten <=15pt] (2n-3) to (2n-2);
       \draw[order] (2n-2) to (2n-1);
       \draw[order] (yn) to (2n-1);
   \begin{pgfonlayer}{background}
    \node[rounded rectangle,draw,dashed,fit=(1)(2)(2n-1),inner xsep=25pt,inner ysep=15pt] {};
    \node[rounded rectangle,draw,dashed,fit=(xn)(yn),inner xsep=14pt,inner ysep=10pt] {};
  \end{pgfonlayer}
 \end{scope}
\end{tikzpicture}
\caption{Structures for Lemma~\ref{lem:gen2}}\label{fig:gen2}
\end{figure}
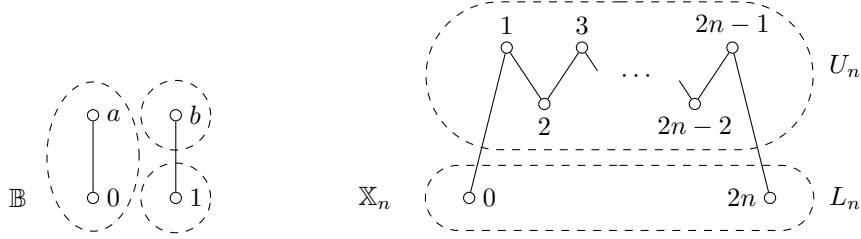

\begin{lemma}\label{lem:gen2}
Let\/ $\A$ be a finite algebra and let\/ $\AT = \langle A; r, s\rangle$ be an alter ego of\/~$\A$, where $r, s \subseteq A^2$. Let $\BT = \langle \{0,a,1\}; \leq, \trianglelefteq \rangle$ be the structure shown in Figure~\ref{fig:gen2}, where $\leq$ is the order and $\trianglelefteq$ is the quasi-order given by
\[
{\leq} = \Delta_B \cup \{(0,a), (1,b)\} \quad\text{and}\quad {\trianglelefteq} = {\leq} \cup \{(a,0)\}.
\]
If\/ $\IScP \AT$ contains the structure $\BT$, then $\A$ admits infinitely many relations.
\end{lemma}

\begin{proof}
Assume that\/ $\IScP \AT$ contains the structure $\BT$.
We use Lemma~\ref{lem:inf} again. For $n \ge 1$, let $\X_n = \langle \{0,1,\dots, 2n\}; \leq, \trianglelefteq\rangle$ be
as in Figure~\ref{fig:gen2}: the ordered set $\langle X_n; \leq \rangle$ is the $(2n+1)$-element fence, and the binary relation~$\trianglelefteq$ is the quasi-order on $X_n$ given by
\[
{\trianglelefteq} = (L_n)^2 \cup (U_n)^2 \cup (L_n \times U_n),
\]
where $L_n := \{0,2n\}$ and $U_n := \{1, 2, \dotsc, 2n-1\}$. Define $\psi_n \colon X_n \to B$ by
\begin{equation*}
 \psi_n(x) =
 \begin{cases}
 1, &\text{if $x=2n$,}\\
 b, &\text{otherwise.}
 \end{cases}
\end{equation*}
Then $\psi_n$ is not a morphism from $\X_n$ to~$\BT$, as $0 \trianglelefteq 2n$ but $\psi_n(0) = b \not\trianglelefteq 1 = \psi_n(2n)$. Since $\BT \in \IScP \AT$, there must be a morphism $\rho_n \colon \BT \to \AT$ such that $\phi_n := \rho_n \circ \psi_n$ is not a morphism from $\X_n$ to~$\AT$.

Using Lemma~\ref{lem:inf}, we can show that $\A$ admits infinitely many relations by establishing the following two claims.

\Claim{Claim 1} {$\X_n \in \IScP\AT$, for all $n \ge 1$.}%
Since $\BT \in \IScP \AT$, it is enough to show that $\X_n \in \IScP \BT$. Let $x, y \in X_n$ with $x \nleq y$. Since $\{0,a\}^2 \subseteq {\trianglelefteq}^{\BT}$, we can define the morphism $\alpha_x \colon \X_n \to \BT$ by
\begin{equation*} 
 \alpha_x(z) =
 \begin{cases}
 a, &\text{if $x \leq z$,}\\
 0, &\text{otherwise,}
 \end{cases}
\end{equation*}
and we have $\alpha_x(x) = a \nleq 0 = \alpha_x(y)$.

To separate the relation~$\trianglelefteq$, we define the morphism $\beta \colon \X_n \to \BT$ by
\begin{equation*} 
 \beta(z) =
 \begin{cases}
 b, &\text{if $z \in U_n$,}\\
 1, &\text{if $z \in L_n$.}
 \end{cases}
\end{equation*}
For $x,y \in X_n$ with $x \not\trianglelefteq y$, we have $x \in U_n$ and $y \in L_n$; so $\beta(x) = b \not\trianglelefteq 1 = \beta(y)$. We have now shown that $\X_n \in \IScP\BT$.

\Claim{Claim 2} {Let\/ $\omega \colon \X_k \to \X_\ell$, where $k < \ell$. Then $\phi_\ell \circ \omega$ is a morphism from $\X_k$ to~$\AT$.}%
The distance between the elements $0$ and $2\ell$ in the ordered-set reduct of $\X_\ell$ is~$2\ell$. Since the ordered-set reduct of $\X_k$ has diameter $2k < 2\ell$, it follows that $\{0,2\ell\} \nsubseteq \omega(X_k)$. So $\omega(X_k) \subseteq X_\ell \comp \{0\}$ or $\omega(X_k) \subseteq X_\ell \comp \{2\ell\}$. From the definition of~$\psi_\ell$, it is easy to see that both maps $\psi_\ell \rest {X_\ell \comp \{0\}}$ and $\psi_\ell \rest {X_\ell \comp \{2\ell\}}$ preserve $\leq$ and~$\trianglelefteq$. Hence $\psi_\ell \circ \omega$ is a morphism from $\X_k$ to~$\BT$, and thus $\phi_\ell \circ \omega = \rho_\ell \circ \psi_\ell \circ \omega$ is a morphism from $\X_k$ to~$\AT$.
\end{proof}

\begin{lemma}\label{lem:a56}
The Ockham algebras with dual spaces $\Y_5$ and $\Y_6$ admit infinitely many relations.
\end{lemma}

\begin{proof}
The Ockham algebras $\A_5 := K(\Y_5)$ and $\A_6 := K(\Y_6)$ are shown in Figure~\ref{fig:dualY56}. We take both algebras to have the same underlying set $A = \{0, a, b, 1\}$. Define the structure $\AT = \langle A; \leq, \trianglelefteq \rangle$ as shown in Figure~\ref{fig:dualY56}, with the order $\leq$ and quasi-order $\trianglelefteq$ given by
\[
{\leq} = \Delta_A \cup \{ (0,a), (1,b) \}
\quad\text{and}\quad
{\trianglelefteq} = {\leq} \cup \{(a,0)\}.
\]
Then it is easy to check that $\AT$ is an alter ego of both $\A_5$ and $\A_6$. So it follows immediately from Lemma~\ref{lem:gen2} that $\A_5$ and $\A_6$ admit infinitely many relations.
\end{proof}

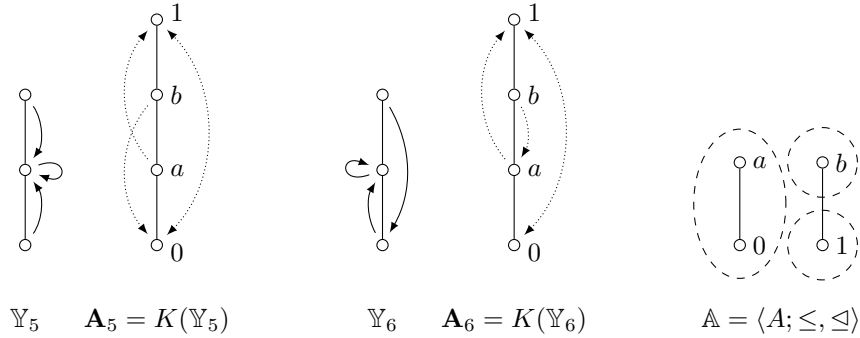
\begin{figure}[t]
\begin{tikzpicture}
  \begin{scope}
   \node at (0,-1) {$\Y_5$};
     \node[unshaded] (0) at (0,0) {};
     \node[unshaded] (1) at (0,1) {};
     \node[unshaded] (2) at (0,2) {};
        \draw[order] (0) to (1);
        \draw[order] (1) to (2);
        \draw[loopy] (1) to [out=20,in=-20] (1);
        \draw[curvy, bend angle=35] (0) to [bend right] (1);
        \draw[curvy, bend angle=35] (2) to [bend left] (1);
          \end{scope}
   \begin{scope}[xshift=1.75cm] 
   \node at (0,-1) {$\A_5 = K(\Y_5)$};
     \node[unshaded] (0) at (0,0) {};
         \node[label, anchor=west] at ($(0)+(0,-0.1)$) {$0$};
      \node[unshaded] (a) at (0,1) {};
         \node[label, anchor=west] at (a) {$a$};
      \node[unshaded] (b) at (0,2) {};
         \node[label, anchor=west] at (b) {$b$};
      \node[unshaded] (1) at (0,3) {};
         \node[label, anchor=west] at ($(1)+(0,0.1)$) {$1$};
      \draw[order] (0) to (a);
      \draw[order] (a) to (b);
      \draw[order] (b) to (1);
      \draw[curvy, <->, densely dotted, bend angle=45] (0) to [bend right] (1);
      \draw[curvy, bend angle=40, densely dotted] (a) to [bend left] (1);
      \draw[curvy, bend angle=40, densely dotted] (b) to [bend right] (0);
         \end{scope}
  \begin{scope}[xshift=4.75cm] 
   \node at (0,-1) {$\Y_6$};
     \node[unshaded] (0) at (0,0) {};
     \node[unshaded] (1) at (0,1) {};
     \node[unshaded] (2) at (0,2) {};
        \draw[order] (0) to (1);
        \draw[order] (1) to (2);
        \draw[loopy] (1) to [out=-160,in=-200] (1);
        \draw[curvy] (0) to [bend left] (1);
        \draw[curvy] (2) to [bend left] (0);
          \end{scope}
   \begin{scope}[xshift=6.5cm] 
   \node at (0,-1) {$\A_6 = K(\Y_6)$};
     \node[unshaded] (0) at (0,0) {};
         \node[label, anchor=west] at ($(0)+(0,-0.1)$) {$0$};
      \node[unshaded] (a) at (0,1) {};
         \node[label, anchor=west] at (a) {$a$};
      \node[unshaded] (b) at (0,2) {};
         \node[label, anchor=west] at (b) {$b$};
      \node[unshaded] (1) at (0,3) {};
         \node[label, anchor=west] at ($(1)+(0,0.1)$) {$1$};
      \draw[order] (0) to (a);
      \draw[order] (a) to (b);
      \draw[order] (b) to (1);
      \draw[curvy, <->, densely dotted, bend angle=45] (0) to [bend right] (1);
      \draw[curvy, bend angle=40, densely dotted] (a) to [bend left] (1);
      \draw[curvy, densely dotted] (b) to [bend left] (a);
         \end{scope}
   \begin{scope}[xshift=9.5cm] 
      \node at (0.55,-1) {$\AT = \langle A; \leq, \trianglelefteq\rangle$};
       \node[unshaded] (0) at (0,0) {};
         \node[anchor=west] at ($(0) + (0.05,0)$) {$0$};
      \node[unshaded] (a) at (0,1.1) {};
         \node[anchor=west] at ($(a) + (0.05,0)$){$a$};
     \node[unshaded] (1) at (1.1,0) {};
         \node[anchor=west] at ($(1) + (0.05,0)$) {$1$};
     \node[unshaded] (b) at (1.1,1.1) {};
         \node[anchor=west] at ($(b) + (0.05,0)$) {$b$};
      \draw[order] (a) to (0);
      \draw[order] (1) to (b);
  \begin{pgfonlayer}{background}
    \node[ellipse,draw,dashed,fit=(0)(a),inner xsep=10pt,inner ysep=2pt] {};
    \node[circle,draw,dashed,fit=(b),inner sep=7pt] {};
    \node[circle,draw,dashed,fit=(1),inner sep=7pt] {};
  \end{pgfonlayer}
 \end{scope}
\end{tikzpicture}
\caption{The Ockham algebras with dual spaces $\Y_5$ and $\Y_6$}\label{fig:dualY56}
\end{figure}

For finite Ockham algebras $\A$ and $\B$ such that $H(\A)$ is a divisor of $H(\B)$, if $\A$ admits infinitely many relations, then $\B$ does too, by Lemmas~\ref{lem:transfer} and~\ref{lem:divisor}. So the implication $(1) \Rightarrow (3)$ of the Main Theorem~\ref{thm:main} now follows from Lemmas~\ref{lem:DP}, \ref{lem:K}, \ref{lem:a2} and~\ref{lem:a56}.

\section{The characterisation for Ockham algebras}\label{sec:char}

In this section, we complete the proof of the Main Theorem~\ref{thm:main} by showing that $(3) \Rightarrow (2)$.

\begin{lemma}
Let\/ $\X$ be a non-empty finite Ockham space. Assume that\/ $\X$ is not isomorphic to any of the Ockham spaces in Figure~\ref{fig:finiteOckham}. Then one of the Ockham spaces in Figure~\ref{fig:infiniteOckham} is a divisor of\/~$\X$.
\end{lemma}

\begin{proof}
Since the Ockham space $\X = \langle X; g, \le, \T \rangle$ is finite, it follows that $\X$ must contain an $n$-cycle, for some $n \ge 1$. We break the proof up into three cases.

\Case{Case 1} {$\X$ contains an even cycle.}%
The Ockham space $\Y_3$ is a divisor of $\X$, by Lemma~\ref{lem:evencycle}.

\Case{Case 2} {$\X$ contains two different odd cycles.}%
Let $C$ and $D$ be disjoint odd cycles of~$\X$. Define $\Z \le \X$ by $Z := C \cup D$. Then $\Z$ is an antichain, by Lemma~\ref{lem:oddcycle}. So we can define the morphism $\alpha \colon \Z \to \Y_1$ by
\begin{equation*}
 \alpha(z) =
 \begin{cases}
 0, &\text{if $z \in C$,}\\
 1, &\text{if $z \in D$.}
 \end{cases}
\end{equation*}
Thus $\Y_1 \in \HS \X$.

\Case{Case 3} {$\X$ contains only one cycle.}%
Let $C$ be the unique cycle of $\X$. By Case~1, we can assume that $m := \abs C$ is odd. So $C$ is an antichain in $\X$, by Lemma~\ref{lem:oddcycle}. Since $\X$ is not isomorphic to $\CT_m$ from Figure~\ref{fig:finiteOckham}, we must have $X \comp C \ne \emptyset$. We consider two subcases.

\Case{Case 3a} {$\X$ is one-generated.}%
There is $x \in X \comp C$ with $g(x) \in C$. If $C \cup \{x\}$ is an antichain, then it is easy to see that $\Y_2 \in \HS \X$. So we can assume without loss of generality that $x < c$, for some $c \in C$. Since $g$ is order-reversing, we get $g(c) \le g(x)$. But $g(x) \in C$ and so $g(c) = g(x)$, as $C$ is an antichain. The substructure of $\X$ on $C \cup \{x\}$ is isomorphic to~$\DT_m$ from Figure~\ref{fig:finiteOckham}. Therefore $X \comp (C \cup \{x\}) \ne \emptyset$. Since $\X$ is one-generated, there is $y \in X \comp (C \cup \{x\})$ with $g(y) = x$.

Since $g$ is order-reversing and $x \notin {\uparrow} C$, it follows that $y \notin {\downarrow} (C \cup \{x\})$. Now define $\Z \le \X$ by $Z := C \cup \{x,y\}$. We can define the morphism $\beta \colon \Z \to \Y_6$ by
\begin{equation*}
 \beta(z) =
 \begin{cases}
 2, &\text{if $z = y$,}\\
 1, &\text{if $z \in C$,}\\
 0, &\text{if $z = x$,}
 \end{cases}
\end{equation*}
and therefore $\Y_6 \in \HS \X$.

\Case{Case 3b} {$\X$ is not one-generated.}%
First, assume that there exists $z\in X$ with $g(z)\notin C$. Then the substructure $\mathbb Z$ generated by $z$ is 
not isomorphic to any of the Ockham spaces in Figure~\ref{fig:finiteOckham}, and so is covered by Case 3a. Thus we can assume that that there are distinct $x, y \in X \comp C$ such that $g(x), g(y) \in C$. Define $\Z \le \X$ by $Z := C \cup \{x,y\}$. If $x \notin {\uparrow} C$ and $y \notin {\uparrow} C$, then without loss of generality $y \nle x$ and we can define $\gamma \colon \Z \to \Y_4$ by
\begin{equation*}
 \gamma(z) =
 \begin{cases}
 2, &\text{if $z \in C$,}\\
 1, &\text{if $z = y$,}\\
 0, &\text{if $z = x$,}
 \end{cases}
\end{equation*}
and so $\Y_4 \in \HS \X$. Similarly, if $x \notin {\downarrow} C$ and $y \notin {\downarrow} C$, then we can show that $\Y_4^\partial \in \HS \X$.

Without loss of generality, we can now assume that $x \notin {\uparrow} C$ and $y \notin {\downarrow} C$. In this case, it is easy to check that $\Y_5 \in \HS \X$.
\end{proof}

This completes the proof of our main theorem.


\end{document}